\documentclass[a4paper,11pt] {amsart}
\usepackage[utf8]{inputenc}
\usepackage[english]{babel}
\usepackage{hyperref}
\usepackage{amsmath}
\usepackage{amsthm}
\usepackage{amsfonts}
\usepackage{amssymb}
\usepackage{graphicx}
\usepackage{a4wide}
\usepackage{hyphenat}
\usepackage{latexsym}
\usepackage{tcolorbox}
\usepackage[all]{xy}
\usepackage{enumerate}
\usepackage{mathrsfs}
\usepackage{xspace, psfrag, setspace, supertabular}
\usepackage{color}

\newcommand{\comm}[1]{}

\reversemarginpar
\numberwithin{equation}{section}

\theoremstyle{plain}
\newtheorem{thm}{Theorem}[section]
\newtheorem{lemma}[thm]{Lemma}
\newtheorem{prop}[thm]{Proposition}
\newtheorem{cor}[thm]{Corollary}

\newtheorem*{claims}{Claim}

\theoremstyle{definition}
\newtheorem{defi}[thm]{Definition}
\newtheorem{rem}[thm]{Remark}
\newtheorem{rems}[thm]{Remarks}
\newtheorem{exa}[thm]{Example}
\newtheorem{nota}[thm]{Notation}
\newtheorem{hypo}[thm]{Hypotheses}

\newtheorem{que}[thm]{Question}

\newcommand{\R}{\mathbb{R}}
\newcommand{\C}{\mathbb{C}}
\newcommand{\T}{\mathbb{T}}
\newcommand{\D}{\mathbb{D}}

\newcommand{\CC}{\mathcal{C}}
\newcommand{\CD}{\mathcal{D}}
\newcommand{\CE}{\mathcal{E}}
\newcommand{\CF}{\mathcal{F}}
\newcommand{\CJ}{\mathcal{J}}
\newcommand{\CH}{\mathcal{H}}
\newcommand{\CL}{\mathcal{L}}
\newcommand{\CM}{\mathcal{M}}
\newcommand{\CR}{\mathcal{R}}

\newcommand{\CW}{\mathcal{W}}

\newcommand{\FD}{\mathfrak{D}}

\newcommand{\al}{\alpha}
\newcommand{\be}{\beta}
\newcommand{\ga}{\gamma}
\newcommand{\de}{\delta}
\newcommand{\om}{\omega}
\newcommand{\la}{\lambda}
\newcommand{\ep}{\varepsilon}
\newcommand{\si}{\sigma}
\newcommand{\La}{\Lambda}
\newcommand{\ka}{\varkappa}

\newcommand{\vp}{\varphi}
\renewcommand{\epsilon}{\varepsilon}

\renewcommand{\ss}{\subset}
\newcommand{\tn}[1]{\textnormal{#1}}
\newcommand{\ol}[1]{\overline{#1}}

\newcommand{\Space}{\CL}
\newcommand{\abs}[1]{| #1 |}
\newcommand{\Abs}[1]{\left| #1 \right|}
\newcommand{\norm}[1]{\left\lVert #1 \right\rVert}

\newcommand{\weak}[1]{\operatorname{Adm}_{#1}^{w}}
\newcommand{\pweak}[1]{\mathcal{C}_{#1}^{w}}
\newcommand{\wt}[1]{\widetilde{#1}}
\newcommand{\wh}[1]{\widehat{#1}}
\newcommand{\Ran}{\operatorname{Ran}}

\date{\today}
\title{Operator Inequalities,  Functional Models and Ergodicity}

\author{Luciano Abadias}
\address{Luciano Abadias \newline
Departamento de Matem\'aticas, \newline
Instituto Universitario de
Matem\'aticas y Aplicaciones, \newline
Universidad de Zaragoza, \newline
50009 Zaragoza, Spain
}
\email{labadias@unizar.es}

\author{Glenier Bello}
\address{Glenier Bello\newline
Departamento de Matem\'aticas,\newline
Universidad Aut\'onoma de Madrid,\newline
Cantoblanco, 28049 Madrid, Spain,\newline
and Instituto de Ciencias Matem\'aticas (CSIC-UAM-UC3M-UCM)}
\email{glenier.bello@uam.es}

\author{Dmitry Yakubovich}
\address{D. V. Yakubovich\newline
Departamento de Matem\'aticas,\newline
Universidad Aut\'onoma de Madrid,\newline
Cantoblanco, 28049 Madrid, Spain\newline
and Instituto de Ciencias Matem\'aticas (CSIC-UAM-UC3M-UCM)}
\email {dmitry.yakubovich@uam.es}

\usepackage{fancyhdr}
\begin{document}

\begin{abstract}
We discuss when an operator, subject to
a rather general inequality in hereditary form,
admits a unitarily equivalent
functional model of Agler type in the reproducing kernel Hilbert space
associated to the inequality. To the contrary to the previous work,
the kernel need not be of Nevanlinna-Pick type. We derive some
consequences concerning the ergodic behavior of the operator.
\end{abstract}

\keywords{dilation; functional model; operator inequality; ergodic properties}

\maketitle

\section{Introduction}

\subsection{Motivation}
Let $\al(t)$ be a function representable by the power series
$\sum_{n=0}^{\infty} \al_n t^n$ in $\D := \{ \abs{t} < 1 \}$,
where  the coefficients $\al_n$ are real numbers,
and let $T \in L(H)$ be a bounded linear operator on a 
Hilbert\footnote{All Hilbert spaces will be assumed to be separable} 
space $H$.
Put
\begin{equation}
\label{ser-alpha}
\al(T^*,T) := \sum_{n=0}^{\infty} \al_n T^{*n} T^n,
\end{equation}
where the series is assumed to converge in the strong operator
topology SOT in $L(H)$.
When $\al$ is a polynomial, the series above is just a finite sum,
and there is no convergence problem.
In particular, when $\al(t) = 1-t$, the right hand side of
\eqref{ser-alpha} is $I-T^*T$, so $T \in L(H)$
is a contraction if and only if $(1-t)(T^*,T) \ge 0$.
In the 1960's Sz.-Nagy and Foias developed a beautiful spectral theory
of contractions (see \cite{NFBK10}) based on the construction of their
functional model.

In his landmark paper \cite{Agl82},
Agler showed that if $T$ has spectrum $\si(T)$ contained in the unit disc
$\D$ and $\al(T^*,T) \ge 0$, then it is natural to model
$T$ by parts of $B \otimes I_\CE$, where $B$ is a
suitable weighted
backward shift and $I_\CE$ is the
identity operator on some auxiliary  Hilbert space $\CE$.
(By \emph{a part of an operator} we mean its restriction to an invariant subspace.)
More generally, when $\si(T) \ss \ol{\D}$, it has been found
in various particular cases
that instead of $B\otimes I_\CE$
one should consider operators of the form $(B\otimes I_\CE)\oplus S$,
where $S$ is an isometry or a unitary operator.
This representation is called \textit{a coanalytic model}.
As Agler proved in \cite{Agl85}, it holds,
in particular, for
\emph{$m$-hypercontractions}, i.e., operators $T \in L(H)$ such that
$(1-t)^j (T^*,T) \ge 0$ for $j=1, \ldots , m$.
Agler's theorem was generalized in~\cite{MV93} by
M\"uller and Vasilescu to tuples of operators.
The first results on Agler model techniques are exposed in the
book~\cite{bookAglMcC2002} by Agler and McCarthy.
In \cite{Olo05}, Olofsson obtained operator formulas for wandering subspaces,
relevant in the models of $m$-hypercontractions. 
His results were generalized by Eschmeier in \cite{Esc18} to tuples of commuting operators, 
and by Ball and Bolotnikov in \cite{BB13} 
to what they call \emph{$\be$-hypercontractions}. 

M\"{u}ller studied the case where $\al = p$ is a polynomial in \cite{Mul88}.
He considers the class $\CC(p)$ of operators $T \in L(H)$ such that $p(T^*,T) \ge 0$.
He proves that any contraction $T\in\CC(p)$ has a coanalytic model
whenever $p(1) = 0, 1/p(t)$ is analytic in $\D$,
and $1/p(\bar{w}z)$ is a reproducing kernel. This last condition is equivalent to the fact
that all Taylor coefficients of $1/p(t)$ at the origin are positive.
M\"{u}ller also considers some operator inequalities for $T$ with infinitely many terms,
with the same property of positivity. This permits him to show that any operator $T$ is
unitarily equivalent to a part of a backward weighted shift with the same spectral radius
(see \cite[Corollary~2.3]{Mul88}).

In \cite{Olo15}, Olofsson deals with the case where $\al$ is not a polynomial.
His assumptions are that $\al$ is analytic on $\D$, does not vanish on $\D$, and $1/\al$ has
positive Taylor coefficients at the origin.
Under this setting, he studies contractions $T$ on $H$
such that $\al(rT^*,rT) \ge 0$ for every $r \in [0,1)$. With more assumptions,
he obtains the coanalytic model for this class of operators.

In \cite{BY18}, the last two authors considered functions $\al$ in the Wiener algebra $A_W$
of analytic functions in the unit disc with summable sequence of Taylor coefficients,
subject to certain conditions. It was assumed that
the series $\sum \al_n T^{*n}T^n$ converges in norm.
The operators studied there turn out to be similar to contractions
(see \cite[Theorem~I]{BY18}).
This will no longer be true in the setting of the present paper
(see Example~\ref{exa B_s a-contr not similar contr}).

In \cite{BY18}, an explicit  
model in the spirit of Sz.-Nagy and Foias model 
was constructed for 
the class of operators considered there.
The roles of the defect operator 
and the defect space were played by
\begin{equation}\label{eq defi D}
D := (\al(T^*,T))^{1/2}, \quad \FD := \ol{DH},
\end{equation}
where the non-negative square root is taken. 

\subsection{Our setting}
Here the operator $D$ and the space $\FD$, 
defined by \eqref{eq defi D} whenever $\al(T^*,T)\ge 0$, 
will also play an important role. 
Recall that now we consider the convergence of \eqref{ser-alpha} in SOT.
As it will be seen from Example~\ref{exa B_s a-contr not similar contr}, 
this is the appropriate convergence in this context.

Our assumptions are the following.

\begin{hypo}
\label{hypo-al}
Suppose $\al$ is a function in $A_W$ which
does not vanish on $\D$. We put
\[
k(t) = 1/ \al(t) = \sum_{n=0}^{\infty} k_n t^n
\qquad t \in \D,
\]
with $\al_0 = k_0 = 1$, and assume that $k_n>0$ for every $n \ge 1$.
\end{hypo}

Under Hypotheses~\ref{hypo-al}, we denote by $\CH_k$
the weighted Hilbert space of power series
$f(t)=\sum_{n=0}^\infty f_n t^n$
with finite norm
\[
\|f\|_{\CH_k} := \bigg(\sum_{n=0}^\infty  |f_n|^2 k_n \bigg)^{1/2}.
\]
Let $B_k$ be the
\emph{backward shift} on $\CH_k$, defined by
\begin{equation}
\label{def-B-k}
B_k f(t)= \dfrac{f(t) - f(0)}{t}.
\end{equation}

\begin{defi}\label{defi Agler-modelable}
Fix a function $\al$ satisfying Hypotheses~\ref{hypo-al},
and let $T$ be an operator in $L(H)$.
We say that $T$ is \emph{$\al$-modelable}
if $T$ is unitarily equivalent 
to a part of an operator of the
form $(B_k\otimes I_\CE)\oplus S$, where $S$ is an isometry.
\end{defi}

We remark that $B_k\otimes I_\CE$ acts on the Hilbert space
$\CH_k\otimes \CE$, which can be identified with the weighted
Hilbert space of $\CE$-valued power series
$f(t)=\sum_{n=0}^\infty f_n t^n$
with norm given by
\[
\|f\|_{\CH_k\otimes \CE} = \bigg(\sum_{n=0}^\infty  \|f_n\|_{\CE}^2 \, k_n \bigg)^{1/2}.
\]
It acts according to the same formula \eqref{def-B-k}.

It is natural to pose the following question.

\begin{que}\label{question}
Given a function $\al$ satisfying Hypotheses~\ref{hypo-al},
give a good sufficient condition for
an operator $T \in L(H)$ to be $\al$-modelable.
\end{que}

One of the strongest results in this direction is contained in the recent papers by 
Bickel, Hartz and McCarthy \cite{BHM18} 
and by Clou\^{a}tre and Hartz \cite{CH18}.
It is stated for spherically symmetric tuples
of operators. 
For the case of a single operator, their result can be formulated as follows.

\begin{thm}
[{\cite[Theorem 1.3]{CH18}}]
\label{thmBCHMc}
Let $\al$ be a function with $\al_0 = 1$ and  $\al_n \le 0$ for all $n\ge 1$.
Suppose that $k = 1/\al$
has radius of convergence $1$, $k_n >0$ for every $n \ge 0$ and
\begin{equation}\label{eq kn / kn+1}
\lim_{n \to \infty}  \frac {k_n}{k_{n+1}}=1.
\end{equation}
Then $B_k$ is bounded, 
and a Hilbert space operator $T$ is $\al$-modelable if and only if 
$\al(T^*,T)\ge 0$.      
\end{thm}

It is easy to see that the hypotheses of Theorem~\ref{thmBCHMc}
imply Hypotheses~\ref{hypo-al}.
This theorem concerns the \emph{Nevanlinna-Pick case}, that is, when $\al_0 = 1$ and $\al_n \le 0$ for $n \ge 1$.
Alternatively, we say that $k$ is a \emph{Nevanlinna-Pick kernel}.
In the recent work \cite{ClHartzSchillo2019}, Clou\^{a}tre, Hartz and Schillo
establish a Beurling--Lax--Halmos theorem for reproducing kernel Hilbert spaces
in the Nevanlinna-Pick context.
We refer the reader to
\cite{DougMisraSarkar2012, Olo08, Sar15-I, Sar15-II}
for more results in the Nevanlinna-Pick case.
In the recent preprint \cite{ET19}, Eschmeier and Toth extend 
previous results by Eschmeier \cite{Esc18} 
to all complete Nevanlinna-Pick kernels, in the context of operator tuples.

\subsection{Main results}
The following result gives 
a new answer to Question~\ref{question}.

\begin{thm}\label{Muller-type-thm}
Assume Hypotheses~\ref{hypo-al}.
If $k \in A_W$, and its Taylor coefficients $\{k_n\}$ satisfy
$k_n^{1/n} \to 1, \sup k_n / k_{n+1} < \infty$ and
\begin{equation}\label{cond-Nik-Embry-fuerte}
\lim_{m\to\infty}\sup_{n\ge 2m} \sum_{m\le  j\le n/2}
\dfrac{k_j k_{n-j}}{k_n}  =0,
\end{equation}
then $B_k$ is bounded, and the operator $T \in L(H)$ is a part of 
$B_k \otimes I_\CE$ (for some Hilbert space $\CE$) 
if and only if both
$\sum \abs{\al_n}T^{*n}T^n$ and $\sum k_n T^{*n}T^n$ converge in SOT
and $\al(T^*,T) \ge 0$.
Moreover, in this case one can take $\CE = \FD$.
\end{thm}

As it will be seen later, the SOT-convergence of $\sum \abs{\al_n}T^{*n}T^n$
implies the the SOT-convergence of $\sum \al_nT^{*n}T^n$.

Notice that in Theorem~\ref{Muller-type-thm}, the isometric part $S$ is unnecessary
(see Theorem~\ref{thm can choose VD}~(ii) below for more information).
This theorem shows that $T$ is $\al$-modelable
in many cases when $k$ is not a Nevanlinna-Pick kernel, and so
Theorem~\ref{thmBCHMc} does not apply. 
Not much about these kernels has been known previously.
Given an integer $N \ge 2$, there are examples of functions $k$ satisfying the hypotheses of
Theorem~\ref{Muller-type-thm} with whatever
prescribed signs of the coefficients $\al_2, \dots, \al_N$
(see Example~\ref{exa alpha positive coeff}).
Note that  $\al_1 = -k_1$ is always negative.

\begin{rem}
Suppose that $k \in A_W$ and the sequence
\[
\left\{ \dfrac{k_n}{k_{n+1}} \left( 1+ \dfrac{1}{n+1} \right)^a \right\}
\]
is increasing for some $a>1$.
Then \eqref{cond-Nik-Embry-fuerte} holds.
This is close to \cite[Proposition~34]{Shields-review}.
Indeed, put $k_j^* := (j+1)^{-a}$, and define $\rho_j := k_j / k_j^*$.
Then our condition reduces to the condition
$\rho_{n+2}/\rho_{n+1} \ge \rho_{n+1}/\rho_{n}$, for all $n$,
which implies that
$\rho_j \rho_{n-j} / \rho_n \le C$, for $0 \le j \le n$.
Since
\[
\dfrac{k_j k_{n-j}}{k_n} = \dfrac{\rho_j \rho_{n-j}}{\rho_n} \, \,  \dfrac{k_j^* k_{n-j}^*}{k_n^*}
\]
and $\{ k_n^* \}$ satisfies \eqref{cond-Nik-Embry-fuerte},
it follows that $\{  k_n \}$ also satisfies \eqref{cond-Nik-Embry-fuerte}.

Hence, for sufficiently regular sequences $\{k_n\}$, 
the condition~\eqref{cond-Nik-Embry-fuerte} 
is rather close to the condition $\sum k_n<\infty$.
It can be added that, in fact, in Theorem~\ref{Muller-type-thm}
$\{k_n\}$ need not be regular; moreover, the quotients $k_n/k_{n+1}$ need not converge
(see Remark~\ref{exa alpha positive coeff 2}).
\end{rem}

The techniques employed in the proof are different from~\cite{CH18}. We use, basically,
a combination of M\"{u}ller's arguments in \cite{Mul88} and Banach algebras techniques.

The above theorems open the question of describing invariant subspaces
of $B_k\otimes I_\CE$ and of constructing a functional model of operators
under the study, which certainly would be interesting.
We do not address this question in this paper.

Given an operator $C:H\to \CE$, where
$\CE$ is an auxiliary Hilbert space, we define
\begin{equation}
\label{eq defi VC}
V_C x(z) = C(I_H - zT)^{-1} x,
\quad x \in H, \quad z \in \D.
\end{equation}
The next result shows that
whenever $T$ is $\al$-modelable, the operator $V_D : H \to \CH_k \otimes \mathfrak{D}$
is a contraction, and we can give an explicit model
for $T$ (that is, give explicitly $\CE$, $S$ and the transform
which sends the initial space into the model space).
First we need to state one more technical hypothesis, whose 
meaning will be clear later. 

\begin{hypo}
\label{hypo-adm}
Let $\al$ be a function satisfying Hypothesis~\ref{hypo-al}. 
Put 
\begin{equation}
\be(t)=\sum_{n\ge0}\be_n t^n, \quad \text{where } \be_n=|\al_n|, 
\end{equation}
and $\ga(t)=\be(t)k(t)$. We assume that $k_n/k_{n+1}\le C'$ 
and $\ga_n\le C''k_n$ for all $n\ge 0$.  
\end{hypo}

The condition $k_n/k_{n+1}\le C'$ is equivalent to boundedness of 
$B_k$. If it holds, then the second condition is satisfied whenever 
there is some $N$ such that either $\al_n\ge 0$ for $n\ge N$, 
or $\al_n\le 0$ for $n\ge N$. 

\begin{thm}[Explicit model]
\label{thm-addendum-to-B-Cl-H-Mc}
Assume Hypotheses~\ref{hypo-al} and~\ref{hypo-adm}. 
Let $T$ be $\al$-modelable. 
Then $\al(T^*,T)\ge 0$, 
$V_D$ is a contraction, and hence we can define
\[
W=(I_H-V_D^*V_D)^{1/2}, \quad \CW= \ol{W H}.
\]
Moreover, $S: \CW \to \CW$, given by $SWx := WTx$, is an isometry and the operator
\[
(V_D,W) : H \to (\CH_k \otimes \mathfrak{D}) \oplus \CW, \qquad (V_D,W)h = (V_D h, Wh)
\]
provides a model of $T$, in the sense that $(V_D,W)$ is isometric and
\[
((B_k \otimes I_\mathfrak{D}) \oplus S)\cdot(V_D,W) = (V_D,W)\cdot T.
\]
\end{thm}

\begin{rem}
Suppose $\al$ satisfies the above two hypotheses, and suppose 
that $T$ is an $\al$-modelable operator, which is given already 
by its model without the isometric part. That is, there is an invariant subspace $L$ of 
an operator $B_k\otimes I_\CE$, acting on $\CH_k\otimes \CE$, 
such that $T$ is the restriction of this operator to $L$.  
Then $\FD=\CE$ (identified with the constant functions in $\CH_k\otimes \CE$), and 
$V_D$ is the identity operator on $L$. This follows from Corollary~\ref{cor-Bk-in-pweak} below. 

Similarly, in the general case, 
if $T$ is a part of an operator $(B_k\otimes I_\CE)\oplus S$, where 
$S$ is an isometry, there is a unitary operator $u$ such that 
the transform $(V_D, uW)$ is just the identity. 
\end{rem}

If it is known that $T$ is $\al$-modelable,
one can ask about the uniqueness of the model. For answering this question,
we need the following definitions.

\begin{defi}
\label{defi minimal model}
Let $\mathcal{L}$ be an invariant subspace of $(B_k\otimes I_\CE)\oplus S$, 
where $S : \CW \to \CW$ is an isometry.
We will say that the corresponding model operator
\[
\big((B_k\otimes I_\CE)\oplus S\big) | \mathcal{L}
\]
is \emph{minimal} if the following two conditions hold.
\begin{enumerate}[\rm (i)]
\item
$\mathcal{L}$ is not contained in
$(\CH_k \otimes \CE') \oplus \mathcal{W}$ for any $\CE' \subsetneqq \CE$.
\item
$\mathcal{L}$ is not contained in $(\CH_k \otimes \CE) \oplus \mathcal{W'}$ 
for any $\CW' \subsetneqq \CW$ invariant by $S$.
\end{enumerate}
\end{defi}

In Remark~\ref{rem to minimal model} we show that the explicit model obtained
in Theorem~\ref{thm-addendum-to-B-Cl-H-Mc} is indeed minimal.

Note that under Hypotheses~\ref{hypo-al}, 
$\al$ is defined on the closed unit disc $\ol\D$
and does not vanish on the interval $[0,1)$. 
Since $\al(0) = \al_0 = 1$, we obtain that $\al(1) \ge 0$.
We distinguish the following two cases. This distinction appears already
in \cite[Subsection 2.3]{CH18} for the Nevanlinna-Pick case.

\begin{defi}
Suppose that $\al$ meets Hypotheses~\ref{hypo-al}.
We will say that
$\al$ is of \emph{critical type}
(or, alternatively, that we have the \emph{critical case})
if $\al(1) = 0$.
If $\al(1) > 0$, we will say that $\al$ is of \emph{subcritical type}
(or, alternatively, that we have the \emph{subcritical case}).
\end{defi}

\begin{thm}[Uniqueness of the minimal model]
\label{thm can choose VD}
Suppose that $\al$ meets Hypotheses~\ref{hypo-al} and 
~\ref{hypo-adm}. 
Let $T$ be an $\al$-modelable operator.  
\begin{itemize}
\item[\tn{(i)}]
In the critical case, the minimal model of $T$ is unique. 
More precisely, the pair of transforms $(V_D,W_0)$, where $W_0=(I-V_D^*V_D):H\to \CW_0$
and $\CW_0:=\overline{\Ran} (I-V_D^*V_D)$, gives rise to a minimal model,
and any minimal model is provided by $(V_C, W)$, where
$C=vD$, $W=wW_0:H\to \CW$ and $v, w$ are unitary isomorphisms.

\item[\tn{(ii)}] 
In the subcritical case,
the minimal model of $T$ is not unique, in general. 
However, there always exists a minimal model given by $V=V_D$, 
in the sense that $V_D : H \to \CH_k \otimes I_\FD$ is an isometry 
such that $(B_k \otimes I_D) V_D = V_D T$. 
Note that in this case the isometry $S$ is absent. 
\end{itemize}
\end{thm}

We remark that there are other works that give answers to 
the above Question \ref{question}. In particular, 
Pott \cite{Pott99} gave a model for operators satisfying two inequalities 
$(1-p)(T^*, T)\ge 0$ and $(1-p)^m(T^*, T)\ge 0$, where 
$p$ is a polynomial with nonnegative coefficients, $m\ge 1$ and 
$p(0)=0$ (this class is a generalization of $m$-hypercontractions). 
In fact, she treats tuples of operators. In \cite{BB13}, 
Ball and Bolotnikov consider a function $\al(t)$ 
in the Wiener algebra such that $k=1/\al$ has positive 
coefficients satisfying $0<\epsilon\le k_n/k_{n+1}\le 1$ for all $n$ 
(so that $B_k$ is a contraction). They show that an operator $T$ 
is $\al$-modelable, with absent isometric part, if and only if 
if $\al(T^*, T)\ge 0$ as well as infinitely many additional inequalities 
hold ($T$ is $\beta$-hypercontractive, where $\beta_n=1/k_n$), and 
$T$ is what they call $\beta$-stable. See \cite{BB13}, Theorem 4.3. 
In \cite{BB13}, Theorem 7.2, Ball and Bolotnikov give 
a model of $T$ in terms of their generalization of the 
characteristic function, which is an infinite family of operator-valued 
functions. 

Whereas these authors treat both subcritical and critical cases, 
Theorem~\ref{Muller-type-thm} only concerns the subcritical case 
(because of the condition $k\in A_W$).

\subsection{Consequences of the model}
\label{subsec conseq}
If an operator $T$ is $\al$-modelable, it is natural to study what consequences
can be derived from the model. Here we obtain two types of consequences: 
\begin{enumerate}
\item 
when the defect operator $D$ has finite rank (that is, $\dim \FD < \infty$), and 
\item 
ergodic consequences when $\al(t) = (1-t)^a$ with $0<a<1$. 
\end{enumerate}

We will use the space $\CR_k=\CH_{\tilde k}$, where $\tilde k_n= 1/k_n$. It is easy to see that
it is the reproducing kernel Hilbert space, corresponding to
the positive definite kernel $k(z,w) := k(\bar{w} z)$.
The pairing $\langle f,g\rangle=\sum f_n\bar g_n$
$(f\in \CH_k, g\in \CR_k)$  makes $\CR_k$ naturally dual to $\CH_\ka$.
In this interpretation, the adjoint operator to $B_k$ is the operator
$g(z)\mapsto zg(z)$, acting on $\CR_k$.

If $\al$ is of subcritical type,
we have the following result related to the Carleson condition.

\begin{thm}
\label{thm-conseq-Carleson}
Let $T$ be an operator similar to a part of $B_k\otimes I_\FD$,
acting on the space $\CH_k\otimes\FD$,
where $\CR_k$ is a Banach algebra and $\FD$ is finite dimensional. 
Suppose that 
\[
\lim_{n \to \infty} \left( \inf_{j\ge 0} \,  \dfrac{k_j}{k_{n+j}} \right)^{1/n} 
= \lim_{n \to \infty} k_n^{1/n} = 1, 
\]
and also that
\begin{equation}\label{eq sum kn bound C epsilon}
\sum_{n=N}^{\infty} k_n \le C N^{-\ep} \qquad \forall N \ge 0,
\end{equation}
for some positive constants $C$ and $\ep$ which do not depend on $N$.
Suppose that the spectrum $\si(T)$ does not cover $\D$. Put
\[
E:= (\ol{\si(T) \cap \D}) \cap \T
\]
and let $\{ l_\nu \}$
denote the lengths of the finite complementary intervals of $E$ (in $\T$).
Then the Lebesgue measure of $E$ is $0$, and the Carleson condition holds:
\[
\sum_{\nu} l_\nu \log \dfrac{2\pi}{l_\nu} < \infty.
\]
\end{thm}

Some of the arguments employed in the proof of this theorem are related with 
the so-called index of an invariant subspace of $\CR_k\otimes\CE$; see 
Section~\ref{section finite defect} for more details. 

In the critical case, an important family of functions $\al$ are those of the form
$\al(t) := (1-t)^a$, for $a>0$. Note that they satisfy Hypotheses~\ref{hypo-al}.
When $a=m$ is a positive integer, it is said that $T \in L(H)$ is an
\emph{$m$-contraction} if $(1-t)^m(T^*,T) \ge 0$, and that $T$ is an
\emph{$m$-isometry} if $(1-t)^m(T^*,T) = 0$.
The papers \cite{BBMP19,BMN10,BSZ18,Gu14,Ryd19} (among
others) study $m$-isometries.
The paper \cite{MajMbSuciu16} is dedicated to a profound study of $2$-isometries.
In~\cite{Gu15}, Gu treats a more general class of $(m,p)$-isometries on Banach spaces, and in 
~\cite{Gu18Korean}, he discusses $m$-isometric tuples of operators on a Hilbert space.
In \cite{CS17JMAA}, Chavan and Sholapurkar study another
interesting class of operators: $T$ is a
\emph{joint complete hyperexpansion of order} $m$ if
$(1-t)^n(T^*,T) \le 0$ for every integer $n \ge m$.
That work, in fact, is devoted to tuples of commuting operators.

Here we introduce the case when the exponent $a$ is not an integer. 
The definitions of $a$-contraction and $a$-isometries are the natural ones:
we say that
$T$ is an \emph{$a$-contraction} if $(1-t)^a(T^*,T) \ge 0$,
and $T$ is an \emph{$a$-isometry} if $(1-t)^a(T^*,T) = 0$.

Note that $\al(t) := (1-t)^a$ is of Nevanlinna-Pick type
when $0<a<1$. In this case, with the help of the model given by Theorem~\ref{thmBCHMc},
we will get the following two ergodic results.

\begin{thm}\label{thm Ca to Cesaro bounded}
If $T$ is an $a$-contraction, with $0<a<1$,
then $T$ is quadratically $(C,b)$-bounded for any $b>1-a$.
\end{thm}
That $T$ is \emph{quadratically $(C,b)$-bounded} 
(where the letter C stands for Ces\`{a}ro) 
means that there exists a constant $c>0$ such that
\[
\sup_{n \ge 0} \, \frac{1}{k^{b+1}(n)} \, \sum _{j=0}^n  k^{b}(n-j) \|T^jx\|^2  \le c \|x\|^2
\qquad (\forall x\in H),
\]
where the numbers $k^{-s}(n)$, called \emph{Ces\`{a}ro numbers}, are defined by
\[
(1-t)^{s} =: \sum_{n=0}^{\infty} k^{-s}(n) t^n.
\]

As we will show (see Example~\ref{exa B_s a-contr not similar contr}), 
for any $a\in (0,1)$, 
the class of $a$-contractions on $H$ is strictly 
larger than the class of contractions. It is obvious that any 
contraction is quadratically $(C,b)$-bounded
(due to the equality $\sum _{j=0}^n  k^{b}(n-j)=k^{b+1}(n)$ for any $b>0$). 
The meaning of the above fact is that some ergodic properties of 
contractions still hold true for $a$-contractions. 

\begin{thm}\label{thm S appears iff}
Let $T$ be an $a$-contraction with $0<a<1$ and let $b>1-a$.
Then the following statements are equivalent.
\begin{enumerate}[\rm (i)]
\item
The isometry $S$ does not appear in the $(1-t)^a$-model of $T$.
\item
For every $x \in H$,
\begin{equation}\label{eq limit C-means are 0 general}
\exists \lim_{n \to\infty} \frac{1}{k^{b+1}(n)}\sum _{j=0}^n  k^{b}(n-j) \|T^jx     \|^2 = 0.
\end{equation}
\item
For every $x \in H$,
\[
\liminf_{n \to \infty} \norm{T^n x} = 0.
\]
\end{enumerate}
\end{thm}

\begin{rem}
For any $a\in (0,1)$, there are $a$-contractions which are not contractions.
This follows from Theorem~\ref{thm B_s a-contraction} below. 
The same holds for $a>1$. Indeed, if $m<a\le m+1$, where $m$ is an integer, then
it is easy to get (see our forthcoming paper~\cite{ABY2})
that any $(m+1)$-isometry $T$ is also an $a$-isometry,
which means that $(1-t)^a(T^*,T)=0$. 
There are $(m+1)$-isometries that are not contractions, and each of them is an
example of this type.
\end{rem}

\subsection{Contents}
The paper is organized as follows. 
In Section \ref{section Preliminaries on classes defined by operator inequalities} 
we introduce two families of operators in $L(H)$ depending on a fixed function 
$\al(t) = \sum_{n\ge 0} \al_n t^n$: $\weak{\al}$ and $\pweak{\al}$.
Essentially, $\weak{\al}$ is the family of operators $T$ for which we can
define $\al(T^*,T)$, and its subfamily $\pweak{\al}$ consists of those $T$ 
for which $\al(T^*,T) \ge 0$.
We use the superscript notation ``$w$'' in $\weak{\al}$ and $\pweak{\al}$
to make it easier to compare the results from~\cite{BY18} and from the present paper.
Notice that in~\cite{BY18}, only the convergence of

We obtain some
interesting properties of these families and characterize the membership of
backward and forward weighted shifts to them.
In Section~\ref{section On the extensions}, we prove
Theorems~\ref{thm-addendum-to-B-Cl-H-Mc} and~\ref{thm can choose VD}.
The proof of Theorem \ref{Muller-type-thm} is given in
Section~\ref{section Proof of Theorem Muller-type-thm}.
In Section~\ref{Analysis of condition Muller type} we study the scope of
Theorem~\ref{Muller-type-thm}. There
we present examples satisfying the hypothesis of Theorem \ref{Muller-type-thm},
where Theorem~\ref{thmBCHMc} does not apply.
In Section~\ref{section finite defect} we prove
Theorem~\ref{thm-conseq-Carleson}.
The proofs of
Theorems~\ref{thm Ca to Cesaro bounded} and \ref{thm S appears iff}
are given in Section \ref{section Consequences for ergodic theory}.

In our forthcoming paper~\cite{ABY2}, we will study
models up to similarity (instead of unitary equivalence).
There we will consider functions $\al$
that may have zeroes in $\D$. We will prove that under certain hypotheses,
any operator in $\pweak{\al}$ is similar to an $a$-contraction
if $\al(t)$ ``behaves like'' $(1-t)^a$ in a neighborhood of $1$.
We will also study  $a$-contractions in more detail.

\section{Preliminaries on classes defined by operator inequalities}
\label{section Preliminaries on classes defined by operator inequalities}

In this section we introduce the operator classes $\weak{\al}$ and $\pweak{\al}$ associated to
a function $\al(t)= \sum_{n\ge 0} \al_n t^n$,
with $\al_n \in \R$.
After studying them, we analyze why
Hypotheses~\ref{hypo-al}
are natural. Finally, at the end of the section we discuss the membership of weighted shifts in
the classes $\weak{\al}$ and $\pweak{\al}$.

\subsection{The classes $\weak{\al}$ and $\pweak{\al}$}
Before entering into the definitions and basic properties of these classes, let us mention
the following well known result that will be used repeatedly.
\begin{lemma}[see {\cite[Problem 120]{Hal82}}]
\label{Halmos-lemma}
If an increasing sequence $\{A_n\}$ of selfadjoint Hilbert space operators satisfies $A_n\le CI$
for all $n$, where $C$ is a constant, then $\{ A_n \}$ converges in the strong operator topology.
\end{lemma}

\begin{defi}\label{defi C alpha weak}
Given a function $\al(t) = \sum_{n\ge 0} \al_n t^n$ with $\al_n \in \R$, we put
\begin{equation}\label{eq. defi of C alpha weak}
\weak{\al} := \bigg\{ T \in L(H) \, : 
\sum_{n=0}^{\infty} |\al_n| \norm{T^n x}^2 <\infty \tn{ for every } x \in H \bigg\}.
\end{equation}
Note that
this class of operators is not affected if we change the signs of some coefficients $\al_n$'s.
\end{defi}

If $X$ and $Y$ are two quantities (typically non-negative),
then $X \lesssim Y$ (or $Y \gtrsim X$) will mean that $X \le CY$ 
for some absolute constant $C > 0$.
If the constant $C$ depends on some parameter $p$,
then we write $X \lesssim_{\, p} Y$. 
We write $X \asymp Y$ when both
$X \lesssim Y$ and $Y \lesssim X$.

\begin{prop}\label{prop equivalences for Adm alpha weak}
The following statements are equivalent.
\begin{enumerate}
\item[\tn{(i)}] 
$T\in \weak{\al}$.
\item[\tn{(ii)}] 
$\sum_{n=0}^{\infty} \abs{\al_n} \norm{T^n x}^2 \lesssim \norm{x}^2$ for every $x \in H$.
\item[\tn{(iii)}] 
The series $\sum_{n=0}^{\infty} \abs{\al_n} T^{*n} T^n$ converges 
in the strong operator topology in $L(H)$.
\end{enumerate}
\end{prop}

\begin{proof}
Suppose that (i) is true. Note that for every $x,y \in H$ and $M > N$ we have
\[
\begin{split}
\Abs{ \sum_{n=N+1}^{M} \abs{\al_n} \langle T^n x, T^n y \rangle } 
&\leq 
\sum_{n=N+1}^{M} \abs{\al_n} \norm{T^n x} \norm{T^n y} \\
&\leq 
\frac{1}{2} \left\{ \sum_{n=N+1}^{M} \abs{\al_n} \norm{T^n x}^2 
+ \sum_{n=N+1}^{M} \abs{\al_n} \norm{T^n y}^2
\right\}
\to 0
\end{split}
\]
as $N$ and $M$ go to infinity.
Therefore
\begin{equation}\label{eq. convergence for x y}
\sum_{n=0}^{\infty} \abs{\al_n} \langle T^n x, T^n y \rangle 
\quad \quad \tn{converges (in } \C \tn{)},
\end{equation}
for every $x,y \in H$. Put
\begin{equation}\label{eq. definition AN}
A_N := \sum_{n=0}^{N} \abs{\al_n} T^{*n}T^n \in L(H)
\end{equation}
for every non-negative integer $N$. 
Fix $x \in H$. 
By \eqref{eq. convergence for x y} we know that 
$\langle A_N x, y \rangle$ converges for every $y \in H$. 
This means that the sequence $\{ A_N x \} \ss H$ is weakly convergent. 
Then $\sup_N \norm{A_N x} < \infty$ for any $x \in H$ 
and therefore $\sup_N \norm{A_N} < \infty$. 
Hence (ii) follows with absolute constant $\sup_N \norm{A_N}$.

Now suppose we have (ii). 
This means that the operators $A_N$ given by \eqref{eq. definition AN} 
are uniformly bounded from above. 
So we can apply Lemma \ref{Halmos-lemma} to obtain (iii).

Finally, it is immediate that (iii) implies (i). This completes the proof.
\end{proof}

\begin{cor}
If $T \in \weak{\al}$, then the series
\[
\al(T^*,T) := \sum_{n=0}^{\infty} \al_n T^{*n} T^n
\]
converges in the strong operator topology in $L(H)$.
\end{cor}

\begin{proof}
Let $T \in \weak{\al}$. By Proposition~\ref{prop equivalences for Adm alpha weak},
the series $\sum \abs{\al_n} T^{*n} T^n$ converges in SOT. Put
\[
\al_n^+
:=
\begin{cases}
\al_n & \text{if } \al_n \ge 0  \\
0 & \text{if } \al_n < 0
\end{cases},
\quad \quad \quad
\al_n^-
:=
\begin{cases}
0 & \text{if } \al_n \ge 0  \\
-\al_n & \text{if } \al_n < 0
\end{cases}.
\]
Hence
\begin{equation}\label{alpha n mas y menos}
\sum_{n=0}^{N} \al_n T^{*n}T^{n} = \sum_{n=0}^{N} \al_n^ + T^{*n}T^{n} 
- \sum_{n=0}^{N} \al_n^- T^{*n}T^{n}.
\end{equation}
It is immediate (using again Lemma \ref{Halmos-lemma}) that 
both sums on the right hand side of \eqref{alpha n mas y menos} have limits in
SOT as $N \to \infty$, and therefore the corollary follows.
\end{proof}

This corollary allows us to introduce the following class of operators in $L(H)$, 
also depending on $\al$.

\begin{defi} Let
\[
\pweak{\al} := \{ T \in \weak{\al} \, : \, \al(T^*,T) \ge 0 \}.
\]
Sometimes, by abuse of notation, 
we will simply write $\al(T^*,T) \ge 0$ instead of $T \in \pweak{\al}$. 
In particular, this means that $T \in \weak{\al}$.
\end{defi}

\begin{prop}\label{prop_immediate_properties_Calpha}
\quad
\begin{enumerate}[\rm (i)]
\item
If $T \in \pweak{\al}$, then any part of $T$ also belongs to $\pweak{\al}$.
\item
If $T_1, T_2 \in \pweak{\al}$,
then $T_1 \oplus T_2 \in \pweak{\al}$.
\item
If $T \in \pweak{\al}$, then $T \otimes I_\CE$
(where $I_\CE$ is the identity operator on some Hilbert space $\CE$)
also belongs to $\pweak{\al}$.
\end{enumerate}
\end{prop}

\begin{proof}
Note that an operator $T\in L(H)$ belongs to $\pweak{\al}$ if and only if
\[
\sum_{n=0}^{\infty} \abs{\al_n} \norm{T^n x}^2 < \infty
\qquad \tn{ and } \qquad
\sum_{n=0}^{\infty} \al_n \norm{T^n x}^2 \ge 0
\]
for every $x \in H$. Then (i) and (ii) are immediate.
For (iii), observe that if $d:=\dim \CE \le \infty$, then the orthogonal sum of
$d$ copies of an operator in $\pweak{\al}$ is clearly in $\pweak{\al}$
(by the Pythagoras Theorem).
\end{proof}

The following proposition will serve us to discuss why our 
Hypotheses~\ref{hypo-al} are natural.

\begin{prop}\label{prop extra condition alpha in AW}
Let $T \in \pweak{\al}$. If $\al \not \in A_W$, then $\si(T) \ss \D$.
\end{prop}

\begin{proof}
Let $T \in \pweak{\al}$, where $\al \not \in A_W$ (that is, $\sum \abs{\al_n} = \infty$). 
By Proposition \ref{prop equivalences for Adm alpha weak} 
we know that there exists a constant $C>0$ such that
\begin{equation}\label{eq unif bdd for norm x le 1}
\sum_{n=0}^{\infty} \abs{\al_n} \norm{T^n x}^2 \le C
\end{equation}
for every $x \in H$ with $\norm{x} = 1$.

Suppose that $T$ has spectral radius $\rho(T) \ge 1$. 
Let $\la$ be any point of $\si(T)$ such that $\abs{\la} = \rho(T)$. 
Then $\la$ belongs to the boundary of the spectrum of $T$ 
and therefore it belongs to the approximate point spectrum. 
Put $R := \abs{\la}^2 = \rho(T)^2 \ge 1$. 
Fix an integer $N$ sufficiently large so that
\[
\sum_{n=0}^{N} \abs{\al_n} > C+1.
\]
Now, choose a unit approximate eigenvector $h \in H$ corresponding to $\la$
such that $\norm{Th-\la h}$ is sufficiently small, so that
\[
\Abs{\norm{T^m h}^2 - \abs{\la}^{2m}} < \left( \sum_{n=0}^{N} \abs{\al_n} \right)^{-1},
\qquad m=0, 1, \ldots , N.
\]
Then
\[
\Abs{\sum_{n=0}^{N} \abs{\al_n} R^n - \sum_{n=0}^{N} \abs{\al_n} \norm{T^n h}^2} 
\le \sum_{n=0}^{N} \abs{\al_n} \Abs{R^n - \norm{T^n h}^2} < 1,
\]
and therefore
\[
\sum_{n=0}^{N} \abs{\al_n} \norm{T^n h}^2 
\ge \bigg(\sum_{n=0}^{N} \abs{\al_n} R^n\bigg) - 1 
\ge \bigg(\sum_{n=0}^{N}\abs{\al_n}\bigg) - 1 > C.
\]
But this contradicts \eqref{eq unif bdd for norm x le 1}. 
Hence $\rho(T)$ must be strictly less that $1$, 
that is, $\si(T) \ss \D$, as we wanted to prove.
\end{proof}

The next result follows immediately imitating the above proof. 
We denote by $r(\al)$ the radius of convergence of the series for $\al$.

\begin{prop}
If $T \in \pweak{\al}$, then $\rho(T)^2 \le r(\al)$.
\end{prop}

One can compare the above two propositions with \cite[Corollary 22]{Stank2013},
which concerns the case when $T$ satisfies an equality $\al(T^*,T) = 0$.

\subsection{Analysis of the Hypotheses~\ref{hypo-al}}
\label{subsection Analysis of the hypothesis}
Observe that Hypotheses~\ref{hypo-al} do not restrict to the Nevanlinna-Pick case.
Let us explain briefly why these hypothesis are natural.

First of all, the assumption that $\al$ belongs to $A_W$
is natural due to Proposition~\ref{prop extra condition alpha in AW}.
To assure that $k = 1/ \al$ is analytic in $\D$,
we need that $\al$ do not vanish in $\D$.
In order to guarantee that we can obtain a reproducing kernel Hilbert space
$\CR_k$ of analytic functions, we need to assume that $k_n > 0$ for every $n \ge 0$.
(See Remark~\ref{rem duality Rk and Hk} below.)
The assumption $k_0 = 1$ is just a normalization of the coefficients.
Finally, note that in Theorem~\ref{Muller-type-thm}, 
which is our new source of examples when compared with 
Theorem~\ref{thmBCHMc}, we need that $k \in A_W$.
However, this assumption excludes automatically
the critical case (when $\al(1) = 0$).
Therefore, it is natural to just make the assumption that $k$ is analytic in $\D$, 
so we can still consider both cases: critical and subcritical.

As we already mentioned in the Introduction,
in~\cite{ABY2} we will drop
the assumption that $\al$ does not vanish on $\D$.

\subsection{The weighted shifts $B_\ka$ and $F_\ka$}
Given  a sequence of positive numbers $\{\ka_n: n\ge 0\}$,
we denote by $\CH_\ka$ the corresponding weighted Hilbert
space of power series $f(t)=\sum_{n=0}^\infty f_n t^n$ with the norm
\[
\|f\|_{\CH_\ka} := \bigg(\sum_{n=0}^\infty  |f_n|^2 \ka_n \bigg)^{1/2}.
\]
Obviously, the monomials $e_n(t) := t^n$, for $n \ge 0$, 
form an orthogonal basis on $\CH_\ka$, and
\begin{equation}\label{eq norm en}
\norm{e_n}_{\CH_\ka}^2 = \ka_n.
\end{equation}
The backward and forward shifts $B_\ka$ and $F_\ka$ on $\CH_\ka$ are defined by
\begin{equation}\label{eq defi B_la and F_la for functions}
B_\ka f(t) := \dfrac{f(t) - f(0)}{t}
\quad \tn{ and } \quad
F_\ka f(t) := t f(t)
\qquad \qquad (\forall f \in \CH_\ka),
\end{equation}
or equivalently
\begin{equation}\label{eq defi B_la and F_la for e_n}
B_\ka e_n :=
\begin{cases}
e_{n-1}, & \text{if}\ n \ge 1 \\
0, & \text{if}\ n =0
\end{cases}
\quad \text{and}
\quad F_\ka e_n := e_{n+1}
\qquad \qquad (\forall n \geq 0).
\end{equation}
It is immediate that
$\norm{B_\ka}^2 = \sup_{n \ge 0} \ka_n/\ka_{n+1}$.
Hence $B_\ka$ is bounded if and only if
\begin{equation}
\label{eq backward bdd}
\dfrac{\ka_n}{\ka_{n+1}} \le C
\quad \quad (\forall n \ge 0),
\end{equation}
for a constant $C>0$.
Analogously,
$\norm{F_\ka}^2 = \sup_{n \ge 0} \ka_{n+1} / \ka_n$, and therefore
$F_\ka$ is bounded if and only if

\begin{equation}
\label{eq forward bdd}
0<c \le \dfrac{\ka_n}{\ka_{n+1}}
\quad \quad (\forall n \ge 0), 
\end{equation}
for some constant $c$.

\begin{rem}
\label{rem duality Rk and Hk}
At the beginning of Subsection~\ref{subsec conseq} we discussed the 
duality of the spaces $\CH_k$ and $\CR_k$, where $k(t)$ was, as usual, 
the function $1/\al(t)$. 
Of course, if we replace $k$ with any other function
$\ka(t)=\sum_{n\ge 0} \ka_n t^n$ (where $\ka_n>0$), 
the same duality will hold for the spaces $\CH_\ka$ and $\CR_\ka$. 
\end{rem}

\begin{nota}\label{Notation Bs}
Let us mention here a convenient notation that will be used in
Section~\ref{section Consequences for ergodic theory}.
When $\{ \ka_n \}$ is precisely the sequence of Taylor coefficients of
the function $(1-t)^{-s}$ for some $s>0$, that is,
\[
\ka_0 = 1
\quad \tn{ and } \quad
\ka_n = \dfrac{s(s+1) \cdots (s+n-1)}{n!}
\quad \tn{ for } n \ge 1,
\]
we denote the space $\CH_\ka$ by $\CH_s$, emphasizing the exponent $s$.
In the same way we use $B_s$ and $F_s$.
\end{nota}

\begin{lemma}\label{lemma T in Cal for backward and forward}
Let $T$ be one of the operators $B_\ka$ or $F_\ka$,
for some $\ka(t)=\sum_{n\ge 0} \ka_n t^n$.
Suppose that $T$ is bounded 
(i.e., assume \eqref{eq backward bdd} or \eqref{eq forward bdd}, respectively).
Then:
\begin{enumerate}
\item[\tn{(i)}]
$T \in \weak{\al}$ if and only if
\begin{equation}\label{eq Adm alpha for basis}
\sup_{m\ge 0} \left\{ \sum_{n=0}^{\infty} \abs{\al_n}
\dfrac{\norm{T^n e_m}^2}{\norm{e_m}^2} \right\} < \infty.
\end{equation}
\item[\tn{(ii)}]
Suppose that $T \in \weak{\al}$. Then $T \in \pweak{\al}$ if and only if
\begin{equation}\label{eq C alpha for basis}
\sum_{n=0}^{\infty} \al_n \norm{T^n e_m}^2 \ge 0 \qquad (\forall m \ge 0).
\end{equation}
\end{enumerate}
\end{lemma}

\begin{proof}
(i) Let $T \in \weak{\al}$. 
By Proposition \ref{prop equivalences for Adm alpha weak} (ii) we have
\[
\sum_{n=0}^{\infty} \abs{\al_n} \norm{T^n f}^2 \lesssim \norm{f}^2,
\]
for every function $f \in \CH_\ka$. 
Taking the vectors of the basis $f = e_m$ we obtain \eqref{eq Adm alpha for basis}.

Conversely,
let us assume now \eqref{eq Adm alpha for basis}. Fix a function $f\in \CH_\ka$. Then
\[
T^n f = \sum_{m=0}^{\infty} f_m T^n e_m
\qquad \qquad (\forall n \ge 0),
\]
where the series is orthogonal. Therefore
\begin{equation}\label{eq Back Forw change summation}
\begin{split}
\sum_{n=0}^{\infty} \abs{\al_n} \norm{T^n f}^2
&= 
\sum_{n=0}^{\infty} \abs{\al_n} \sum_{m=0}^{\infty} \abs{f_m}^2 \norm{T^n e_m}^2
= \sum_{n=0}^{\infty} \abs{\al_n} \sum_{m=0}^{\infty} \abs{f_m}^2 
\norm{e_m}^2 \frac{\norm{T^n e_m}^2}{\norm{e_m}^2} \\
&= 
\sum_{m=0}^{\infty} \abs{f_m}^2 \norm{e_m}^2 
\sum_{n=0}^{\infty} \abs{\al_n} \frac{\norm{T^n e_m}^2}{\norm{e_m}^2}
\lesssim 
\sum_{m=0}^{\infty} \abs{f_m}^2 \norm{e_m}^2 < \infty,
\end{split}
\end{equation}
where \eqref{eq Adm alpha for basis} allows us to justify the change of 
the summation indexes in the last equality. Hence $T \in \weak{\al}$.

(ii) Let $T \in \weak{\al}$. If $T \in \pweak{\al}$, then obviously 
\eqref{eq C alpha for basis} follows. 
For the converse implication, note that similarly to
\eqref{eq Back Forw change summation} we get
\[
\sum_{n=0}^{\infty} \al_n \norm{T^n f}^2
=
\sum_{m=0}^{\infty} \abs{f_m}^2 \norm{e_m}^2 \sum_{n=0}^{\infty} \al_n
\frac{\norm{T^n e_m}^2}{\norm{e_m}^2},
\]
so \eqref{eq C alpha for basis} implies that $T \in \pweak{\al}$.
\end{proof}

Writing down this lemma for $B_\ka$ and $F_\ka$ separately,
we immediately get the next two results.

\begin{thm}\label{thm characterization backward in Calpha}
Let $\ka(t)=\sum_{n\ge 0} \ka_n t^n$, such that the coefficients $\{ \ka_n \}$ satisfy
\eqref{eq backward bdd}.
Set $\be(t) = \sum_{n\ge 0} \be_n t^n$ with $\be_n = \abs{\al_n}$. 
Put $\ga(t) = \be(t) \ka(t)$. Then:
\begin{enumerate}
\item[\tn{(i)}] 
$B_\ka \in \weak{\al}$ if and only if
\[
\sup_{m\ge 0} \left\{ \dfrac{\ga_m}{\ka_m} \right\} < \infty.
\]
\item[\tn{(ii)}] 
Suppose that $B_\ka \in \weak{\al}$. 
Then $B_\ka \in \pweak{\al}$ if and only if all the Taylor coefficients of 
$\al(t) \ka(t)$ are non-negative.
\end{enumerate}
\end{thm}

The next statement explains the meaning of Hypotheses~\ref{hypo-adm}. 

\begin{cor}
\label{cor-Bk-in-pweak}
Suppose that Hypotheses~\ref{hypo-adm} hold, and let $T=B_k\otimes I_\CE$. Then 

\begin{enumerate}
\item[\tn{(i)}] $T\in\pweak{\al}$; 

\item[\tn{(ii)}] $\al(T^*, T)f =f_0$, $f=\sum f_n z^n\in \CH_k\otimes \CE$;  

\item[\tn{(iii)}] The corresponding operator 
$V_D:\CH_k\otimes \CE \to \CH_k\otimes \CE$ is the identity operator. 
\end{enumerate}
\end{cor}

Indeed, consider first the case when $T=B_k$. 
Hypotheses~\ref{hypo-adm} imply that 
$B_k\in\weak{\al}$. 
Equality 
$\al(t)k(t)=1$ gives that 
$\sum_{n=0}^{\infty} \al_n\|T^n e_0\|^2=1$ and 
$\sum_{n=0}^{\infty} \al_n\|T^n e_m\|^2=0$  
if $m\ge 1$. 
This implies (ii) for this case. 
In particular, $\al(T^*, T)\ge 0$ (that is, $B_k\in\pweak{\al}$),  
and (iii) follows. 
Finally, the operator $T=B_k\otimes I_\CE$ can be seen as an orthogonal sum 
of $\dim \CE$ copies of $B_k$, which gives the general case.

\medskip 

Before restating Lemma \ref{lemma T in Cal for backward and forward} 
for the forward shift $F_\ka$, we need to introduce some notation.

\begin{nota}
Given a sequence of real numbers
$\La = \{ \La_m \}_{m \ge 0}$, we denote by
$\nabla \La$ the
sequence whose $m$-th term, for $m\ge 0$, is given by
$(\nabla \La)_m = \La_{m+1}$.
In general, if $\be(t) = \sum \be_n t^n$ is an analytic function,
we denote by $\be(\nabla)\La$
the sequence whose $m$-th term is given by
\[
(\be(\nabla)\La)_m=\be(\nabla)\La_m := \sum_{n=0}^{\infty} \be_n \La_{m+n}, 
\]
whenever the series on the right hand side converges for every $m \ge 0$.
\end{nota}

\begin{thm}\label{thm characterization forward in Calpha}
Let $\ka(t)=\sum_{n\ge 0} \ka_n t^n$ be a function such that the coefficients 
$\{ \ka_n \}$ satisfy
\eqref{eq forward bdd}.
Set $\be(t) = \sum_{n\ge 0} \be_n t^n$ with $\be_n = \abs{\al_n}$.
Then:
\begin{enumerate}
\item[\tn{(i)}] 
$F_\ka \in \weak{\al}$ if and only if
\[
\sup_{m\ge 0} \left\{ \dfrac{\be(\nabla) \ka_m}{\ka_m} \right\} < \infty.
\]
\item[\tn{(ii)}] 
Suppose that $F_\ka \in \weak{\al}$. 
Then $F_\ka \in \pweak{\al}$ if and only if $\al(\nabla) \ka_m \ge 0$ 
for every $m \ge 0$.
\end{enumerate}
\end{thm}

\section{Explicit model and its uniqueness}
\label{section On the extensions}

In this section we prove
Theorems~\ref{thm-addendum-to-B-Cl-H-Mc} and~\ref{thm can choose VD}.
Let us start by proving that the operator $V_D$ is
a contraction in the Nevanlinna-Pick case. 
Notice that in the following theorem, we do not require that $\al$ belongs to $A_W$.

\begin{thm}\label{thm V_D is a contraction}
Let $\al(t) := \sum_{n\ge 0} \al_n t^n$, with
$\al_0 = 1$ and $\al_n \le 0$ for $n \ge 1$.
If $T \in L(H)$ satisfies $\al(T^*,T) \ge 0$,
then the operator $V_D$ is a contraction.
\end{thm}

\begin{proof}
Recall that $D^2 = \al(T^*,T)$. Therefore
\[
\norm{Dx}^2 = \sum_{m=0}^{\infty} \al_m \norm{T^m x}^2
\]
for every $x \in H$. Hence
\[
\norm{DT^n x}^2 = \sum_{m=0}^{\infty} \al_m \norm{T^{m+n} x}^2
\]
for every $x \in H$ and every non-negative integer $n$. Fix a positive integer $N$. Then
\[
\begin{split}
\sum_{n=0}^{N} k_n \norm{DT^n x}^2
&=
\sum_{n=0}^{N} k_n \sum_{m=0}^{\infty} \al_m \norm{T^{m+n} x}^2 \\
& =
\sum_{j=0}^{\infty} \left( \sum_{n+m = j, \,  n \le N} k_n \al_m \right) 
\norm{T^j x}^2 =: \sum_{j=0}^{\infty} \tau_j \norm{T^j x}^2.
\end{split}
\]
Since $\al k = 1$ we get $\tau_0 = 1$ and $\tau_1 = \cdots = \tau_N = 0$. 
Moreover,
\[
\tau_{N+i} = k_0 \al_{N+i} + \cdots + k_N \al_i < 0
\]
for every $i \ge 1$, because all the $\al_j$'s above are negative or zero 
and the $k_j$'s are positive. 
Therefore
\[
\sum_{n=0}^{N} k_n \norm{DT^n x}^2 \le \norm{x}^2
\]
for every $N$ and hence the series $\sum k_n \norm{DT^n x}^2$ 
converges for every $x \in H$. This gives
\[
\norm{V_D x}^2 = \sum_{n=0}^{\infty} k_n \norm{DT^n x}^2 \le \norm{x}^2,
\]
as we wanted to prove.
\end{proof}

The following fact is simple and well-known .

\begin{prop}\label{prop-gen-form-intertwining-op}
Let $T\in L(H)$ with $\si(T) \ss \ol{\D}$, and let $\CE$ be a Hilbert space.
A bounded transform $V: H\to \CH_\ka\otimes \CE$ satisfies
\begin{equation}
\label{intertw}
VT= (B_\ka \otimes I_\CE) V
\end{equation}
if and only if there is a bounded linear operator
$C:H\to \CE$ such that $V=V_C$ (see \eqref{eq defi VC}).
\end{prop}

\begin{proof}
It is well-known (and straightforward) that any bounded transform
$V_C$ satisfies~\eqref{intertw}.
Conversely, suppose that $VT = (B_\ka \otimes I_\CE) V$.
Define $a_n(x)$ by
\[
Vx(z) := \sum_{n=0}^{\infty}  a_n(x)z^n , \qquad x \in H.
\]
Then
\[
\sum_{n=0}^{\infty} a_n(Tx) z^n = VTx 
= (B_\ka \otimes I_\CE) Vx = \sum_{n=0}^{\infty} a_{n+1}(x) z^n.
\]
Therefore $a_{n+1}(x) = a_n(Tx)$. The statement follows, putting
$C:=a_0$, which has to be a bounded linear operator.
\end{proof}

\begin{prop}
\label{prop-charzn-intertwining}
Let $C:H\to \CE$ be a bounded operator
and let $T\in \CC_\al^w$.
Then there exists a bounded operator $W:H \to \mathcal{W}$
such that the operator $(V_C, W)$ is isometric and transforms
$T$ into a part of the operator
$(B_k\otimes I_\CE)\oplus S$, where $S \in L(\mathcal{W})$ is an isometry,
if and only if the following conditions hold.
\begin{itemize}
\item[\tn{(i)}]
$V_C:H\to \CH_k\otimes \CE$ is a contraction.

\item[\tn{(ii)}]
For every $x\in H$,
\[
\|x\|^2-\|V_C x\|^2=
\|Tx\|^2-\|V_C Tx\|^2.
\]
\end{itemize}
\end{prop}

\begin{proof}
Let us suppose first the existence of such operator $W$.
Since $(V_C,W)$ is an isometry, (i) holds.
Notice that (ii) is equivalent to proving that 
$\norm{Wx}^2 = \norm{WTx}^2$ for every $x \in H$. 
But this is also immediate since $SWx = WTx$ and $S$ is an isometry.

Conversely, suppose now that (i) and (ii) are true.
By (i), we can put 
$W := (I - V_C^{*} V_C)^{1/2}$ and $\mathcal{W} := \ol{\tn{Ran}} \,  W$. 
Using (ii) we have
\begin{equation}\label{eq norm Wx}
\norm{Wx}^2 = \norm{x}^2 - \norm{V_Cx}^2 
= \norm{Tx}^2 - \norm{V_CTx}^2 = \norm{WTx}^2.
\end{equation}
We define
\[
S(Wx) := WTx,
\]
for every $x \in H$. 
Note that $S$ is well defined, since $\norm{SWx} = \norm{Wx}$ by \eqref{eq norm Wx}.  
Since $WH$ is dense in $\mathcal{W}$, $S$ can be
extended to an isometry on $\mathcal{W}$. By the definition of $W$,
 we know that $(V_C,W)$ is an isometry and it is immediate that
\[
(B_k \otimes I_\mathcal{D}) V_C = V_C T \quad \tn{ and } \quad SW = WT.
\]
This completes the converse implication.
\end{proof}

\begin{prop}
\label{prop-rel-D-C-W}
Let $T\in \CC_\al^w$.
Assume that $C:H\to \CE$ and
$W:H\to \CW$ are any bounded operators such that
$(V_C, W)$ is isometric on $(\CH_k \otimes \CE) \oplus \CW$
and transforms
$T$ into a part of $(B_k\otimes I_\CE)\oplus S$, where $S\in B(\CW)$ is an isometry.
Then $C$ and $D$ are related by
\begin{equation}\label{eq rel-D-C-W}
\|Dx\|^2=\|Cx\|^2+\al(1) \|Wx\|^2, \quad \forall x\in H.
\end{equation}
\end{prop}

\begin{proof}
Since $(V_C, W)$ is isometric, we have
\begin{equation}\label{eq VC W isometry}
\norm{x}^2 = \norm{V_C x}^2 + \norm{Wx}^2 
= \sum_{n=0}^{\infty} k_n \norm{CT^n x}^2 +  \norm{Wx}^2,
\end{equation}
for every $x \in H$.
Substituting $x$ by $T^j x$ above and multiplying by $\al_j$, we obtain that
\[
\begin{split}
\al_j \norm{T^j x}^2
&=
\sum_{n=0}^{\infty} \al_j k_n \norm{CT^{n+j} x}^2 +  \al_j \norm{WT^j x}^2 \\
&=
\sum_{n=0}^{\infty} \al_j k_n \norm{CT^{n+j} x}^2 +  \al_j \norm{Wx}^2,
\end{split}
\]
where we have used that $\norm{Wx}^2 = \norm{WTx}^2$. Therefore
\[
\begin{split}
\norm{Dx}^2
&=
\sum_{j=0}^{\infty} \al_j \norm{T^j x}^2
=
\sum_{j=0}^{\infty} \sum_{n=0}^{\infty} \al_j k_n \norm{CT^{j+n}x}^2 
+ \left( \sum_{j=0}^{\infty} \al_j \right) \norm{Wx}^2 \\
&\stackrel{(\star)}{=}
\sum_{m=0}^{\infty} \left(\sum_{j+n = m}  \al_j k_n \right) \norm{CT^m x}^2 + \al(1) \norm{Wx}^2.
\end{split}
\]
Since $\al k = 1$, the only non-vanishing summand in the last series above is 
for $m=0$ and we obtain \eqref{eq rel-D-C-W}.
Finally, note that the rearrangement in ($\star$) is correct as
\[
\sum_{j=0}^{\infty} \sum_{n=0}^{\infty} \abs{\al_j} k_n \norm{CT^{n+j} x}^2 
\le \sum_{j=0}^{\infty} \abs{\al_j} \norm{T^j x}^2 < \infty,
\]
where we have used \eqref{eq VC W isometry} and that $T\in \CC_\al^w$.
\end{proof}

Recall the definition of the \emph{minimal model}
(Definition~\ref{defi minimal model}).

\begin{rem}\label{rem to minimal model}
Suppose that $T$ is $\al$-modelable.
Then $T$ is unitarily equivalent to $\big((B_k\otimes I_\CE)\oplus S\big) | \mathcal{L}$,
where $\mathcal{L} = \ol{\Ran} \, (V_C,W)$. This model is minimal if and only if
\begin{itemize}
\item[(a)] 
$\ol{\Ran} \, C = \CE$; and
\item[(b)] 
$\ol{\Ran} \, W = \mathcal{W}$.
\end{itemize}
Indeed, in this case, it is easy to see that (a) is equivalent to (i), 
and (b) is equivalent to (ii) in Definition~\ref{defi minimal model}.
\end{rem}

\begin{proof}[Proof of Theorem~\ref{thm can choose VD}]
Suppose that the hypotheses are satisfied. 	
First we notice that $\al(T^*,T)\ge 0$, as it follows from 
Corollary~\ref{cor-Bk-in-pweak} and Proposition~\ref{prop_immediate_properties_Calpha}. 
Therefore $D$ is well-defined. 	

(i) In the critical case (i.e., $\al(1) = 0$), \eqref{eq rel-D-C-W} gives
\[
\norm{Dx} = \norm{Cx} \qquad \forall x \in H,
\]
so there exists a unitary operator $v$ such that $C=vD$. This implies the statement.

(ii) Suppose we are in the subcritical case (i.e., $\al(1) > 0$).
First, we
remark that the model is not unique in general. For instance, take $T=U$ any unitary operator.
Using Proposition~\ref{prop equivalences for Adm alpha weak} and that $\al \in A_W$, 
we obtain that $T \in \weak{\al}$. 
Since
\[
\sum_{n=0}^{\infty} \al_n \norm{T^n x}^2 
= \left( \sum_{n=0}^{\infty} \al_n \right) \norm{x}^2  
\qquad \forall x \in H,
\]
we get that $\al(T^*,T)=\al(1)I \ge 0$.  Obviously, $T = U$ is a minimal model for $T$
(where $\CE = 0$, and $\CW = H$).
Moreover, if $k = 1/ \al$ fits \eqref{cond-Nik-Embry-fuerte},
then Theorem~\ref{Muller-type-thm}
(which is proved in the next section, but its proof is completely independent)
gives another model for $T$.
(See Example~\ref{exa alpha positive coeff} and
Remark~\ref{exa alpha positive coeff 2}.)

Now suppose that
$T$ is any $\al$-modelable operator and
$(V_C,W)$ provides its model.
Let us see that there exists a minimal model of $T$ with $V = V_D$ and $W$ absent.
Changing $x$ by $T^n x$ in \eqref{eq rel-D-C-W} we obtain
\[
\norm{DT^n x}^2 = \norm{C T^n x}^2 + \al(1) \norm{Wx}^2,
\]
where we have used that $\norm{WTx} = \norm{Wx}$. Therefore
\[
\begin{split}
\norm{V_D x}^2
&=
\sum_{n=0}^{\infty} k_n \norm{DT^n x}^2
=
\sum_{n=0}^{\infty} k_n \norm{CT^nx}^2 + k(1)\al(1) \norm{Wx}^2 \\
&=
\norm{V_C x}^2 + \norm{Wx}^2 = \norm{x}^2,
\end{split}
\]
so $V_D : H \to \CH_k \otimes \CE$ is an isometry and
therefore provides a model of $T$.
The space $\mathcal{L}$ is just 
$\Ran V_D$ in $\CH_k \otimes \mathfrak{D}$ (which is closed).
This model is minimal, because $\Ran D$ is dense in $\mathfrak{D}$.
(See Remark \ref{rem to minimal model}.)
This gives all statements of (ii).
\end{proof}

\begin{proof}[Proof of Theorem \ref{thm-addendum-to-B-Cl-H-Mc}]
It is an immediate consequence of Theorem \ref{thm can choose VD} 
and Proposition \ref{prop-charzn-intertwining} (i) that $V_D$ is a contraction. 
Finally, for proving that $(V_D,W)$ gives a model, 
we just need to use the same argument employed in the reciprocal implication of 
Proposition~\ref{prop-charzn-intertwining}.
\end{proof}

Results close to Theorems~\ref{thm-addendum-to-B-Cl-H-Mc} and 
\ref{thm can choose VD} appear in Schillo's PhD thesis \cite{Schillo-tesis}. 
He deals with the generality of tuples of commuting operators, 
but for the case of one operator, 
the hypotheses needed there are more restrictive than ours. 

For example, in \cite[Theorem~5.16]{Schillo-tesis}, 
the uniqueness of the coextension is proved 
when $T$ is what he calls a \emph{strong $k$-contraction}.  
For one single operator $T$ and using our notations, these are operators such that 
$\al(T^*,T) \ge 0$, the limit
\[
\Sigma(T) := I_H - \lim_{N\to\infty} \,\sum_{n=0}^{N} k_n T^{*n} \al(T^*,T) T^n
\]
exists (in SOT), 
$\Sigma(T) \ge 0$, and $\Sigma(T) = T^* \Sigma(T) T$. 
In \cite[Corollary~5.17]{Schillo-tesis}, he gives an explicit model involving the 
defect space $\FD_T$. His assumptions are somewhat technical 
(see \cite[Assumption~5.8]{Schillo-tesis}). 
He also assumes the existence of $\al(B_k^*,B_k)$, for which 
\cite[Proposition~2.10]{Schillo-tesis} says that a sufficient condition
is that the coefficients $\{ \al_n \}$ of the function $\al$ have eventually the same sign. 

Recall that our Theorem~\ref{thm can choose VD}~(ii) says that for the subcritical case 
the model is not unique in general. Therefore, since Schillo obtains uniqueness 
of the coextension, it seems that his assumptions exclude the subcritical case.

Schillo's thesis also contains a result on the description of invariant subspaces of 
a backward shift, 
analogous to $B_k\otimes I_\CE$, in his setting of operator tuples. 

Notice that in Theorems~\ref{thm-addendum-to-B-Cl-H-Mc} and 
\ref{thm can choose VD} we are only assuming that $T$ is $\al$-modelable. 
In particular, we do not impose any restriction about the signs of the Taylor coefficients 
of the function $\al$. 

\section{Proof of Theorem \ref{Muller-type-thm}}
\label{section Proof of Theorem Muller-type-thm}

In this section we prove Theorem~\ref{Muller-type-thm}. 
For that, we need to cite some results concerning Banach algebras.

For any sequence $\om = \{ \om_n \}_{n=0}^{\infty}$ of positive weights,
define the weighted space
\[
\ell^{\infty}(\om) := \left\{ f(t) =
\sum_{n=0}^{\infty} f_n t^n \, : \, \sup_{n \ge 0} \abs{f_n} \om_n < \infty \right\}.
\]
In general, its elements are formal power series.
We will also use the separable version of this space:
\[
\ell^{\infty}_0(\om) := \left\{ f(t) =
\sum_{n=0}^{\infty} f_n t^n \, : \, \lim_{n \to\infty} \abs{f_n} \om_n =0 \right\}.
\]

\begin{prop}[see \cite{Nik1970}]\label{prop Aom Banach alg iff condition series}
$\ell^{\infty}(\om)$ is a Banach algebra
(with respect to the formal multiplication of power series)
if and only if
\begin{equation}\label{eq charact Banach algebra omega_n}
\sup_{n \ge 0} \,  \sum_{j=0}^{n} \dfrac{\om_n}{\om_j \om_{n-j}} < \infty.
\end{equation}
\end{prop}

\begin{thm}\label{thm omega bdd by summable implies algebra}
Let $\om_n > 0$ and $\om_n^{1/n} \to 1$. If $\sup_n \om_{n+1}/\om_{n}<\infty$ and
\begin{equation}\label{cond-Nik-ell-1}
\lim_{m\to\infty}\, \sup_{n\ge 2m} \sum_{m\le  j\le n/2}
\dfrac{\om_n}{\om_j \om_{n-j}}  =0,
\end{equation}
then the following is true.
\begin{itemize}
\item[\tn{(i)}] 
$\ell^\infty(\om)$ is a Banach algebra.
\item[\tn{(ii)}] 
If $f \in \ell^\infty(\om)$ does not vanish on $\ol{\D}$, then $1/f \in \ell^\infty(\om)$.
\end{itemize}
\end{thm}

\begin{proof}
The hypotheses imply~\eqref{eq charact Banach algebra omega_n}, so that (i) follows
from Proposition~\ref{prop Aom Banach alg iff condition series}.
To get (ii), we apply the results of the paper~\cite{FNikZarr98} 
by El-Fallah, Nikolski and Zarrabi.
We use the notation of this paper.
Put $\om'(n)=\om(n)/(n+1)$, $A=\ell^\infty(\om)$ and $A_0=\ell^\infty_0(\om)$.
The hypotheses imply that $A$ (and hence $A_0$) is compactly
embedded into the multiplier convolution algebra $\operatorname{mult}(\ell^\infty(\om'))$,
see~\cite[Lemma 3.6.3]{FNikZarr98}.
Hence, by~\cite[Theorem 3.4.1]{FNikZarr98},
for any $f \in A_0$, $\de_1(A_0, \frak{M}(A_0))=0$,
see \cite[Subsection 0.2.3]{FNikZarr98} for the definition of this quantity.
This means that for any $\de>0$ there is a constant
$c_1(\de)<\infty$ such that the conditions
$f \in A_0$, $\|f\|_A=1$ and $|f|>\de$ on $\ol{\D}$
imply that $1/f \in A_0$ and $\|1/f\|_{A}\le c_1(\de)$.
In particular, (ii) holds for $f$ in $A_0$. To get (ii) in the general case,
suppose that $f \in A$ and $|f|>\de>0$ on $\ol{\D}$.
Since $f(rt)\in A_0$ for all $r<1$, we get that the norms of
the functions $1/f(rt)$ in $A$ are uniformly bounded by $c_1(\de)$ for
all $r<1$. When $r\to 1^-$,
each Taylor coefficient of $1/f(rt)$ tends to the corresponding Taylor coefficient of
$1/f(t)$. It follows that $1/f$ is in $A$ (and $\|1/f\|_A\le c_1(\de)$).
\end{proof}

\begin{proof}[Proof of Theorem \ref{Muller-type-thm}]
Put
\[
\om_n := 1/k_n.
\]
The first part of Theorem~\ref{Muller-type-thm}
(that $B_k$ is bounded) is straightforward.
Also, by Theorem~\ref{thm omega bdd by summable implies algebra}~(i),
$\ell^\infty(\om)$ is an algebra.

First suppose that $T$ is a part of $B_k \otimes I_\CE$, and let us prove
that $B_k\in \pweak{\al} \cap \weak{k}$.

By Theorem \ref{thm characterization backward in Calpha} (i),
we know that $B_k \in \weak{k}$ if and only if
\[
\sum_{j=0}^{m} k_j k_{m-j} \,  \lesssim \,  k_m,
\]
which follows from Theorem \ref{thm omega bdd by summable implies algebra} (i) 
and Proposition \ref{prop Aom Banach alg iff condition series}.

Now let us see that $B_k \in \pweak{\al}$.
By Theorem \ref{thm omega bdd by summable implies algebra} (ii), $\al = 1/k$ belongs
to $\ell^\infty(\om)$, and therefore $\abs{\al_n} \lesssim k_n$. 
Then, since $B_k \in \weak{k}$, we obtain
that $B_k \in \weak{\al}$. 
Finally, Theorem \ref{thm characterization backward in Calpha} (ii) 
gives that $B_k \in \pweak{\al}$
(because $\al k = 1$ has non-negative Taylor coefficients).
Hence $T$ also is in $\pweak{\al}\cap\weak{k}$.

Conversely, let us assume now that $T \in \pweak{\al}\cap\weak{k}$.
We want to prove that $T$ is a part of $B_k \otimes I_\CE$.
We adapt the argument of \cite[Theorem 2.2]{Mul88}
(where the convergence of the series of operators is in the uniform
operator topology).

By Proposition~\ref{prop-gen-form-intertwining-op},
\begin{equation}\label{eq BkV equals VT}
(B_k \otimes I_\mathfrak{D}) V_D = V_DT.
\end{equation}
Moreover,
\[
\begin{split}
\norm{V_Dx}^2 &= \sum_{n=0}^{\infty} k_n \norm{DT^nx}^2
= \sum_{n=0}^{\infty} k_n \sum_{m=0}^{\infty} \al_m \norm{T^{n+m} x}^2 \\
&= \sum_{j=0}^{\infty} \bigg( \sum_{n+m=j} k_n \al_m \bigg) \norm{T^j x}^2 = \norm{x}^2,
\end{split}
\]
where we have used that $\sum_{n+m=j} k_n \al_m$ is equal to $1$ if $j=0$ 
and is equal to $0$ if $j \ge 1$. 
The re-arrangement of the series is correct since, using that 
$T \in \weak{\al}\cap\weak{k}$, we have
\[
\sum_{n=0}^{\infty} k_n \sum_{m=0}^{\infty} \abs{\al_m} \norm{T^{n+m} x}^2
\lesssim
\sum_{n=0}^{\infty} k_n \norm{T^n x}^2
\lesssim
\norm{x}^2
\]
and the series converges absolutely.

Hence $V_D$ is an isometry. Joined to \eqref{eq BkV equals VT},
this proves that $T$ is unitarily equivalent to a part of
$B_k \otimes I_\mathfrak{D}$.
\end{proof}

Notice that in particular, we showed that the hypotheses of 
Theorem~\ref{Muller-type-thm} imply Hypotheses~\ref{hypo-adm}.

\section{Discussion of Theorem~\ref{Muller-type-thm}}
\label{Analysis of condition Muller type}

In this section we discuss the scope of Theorem~\ref{Muller-type-thm} and
give a series of examples where it applies, whereas Theorem~\ref{thmBCHMc} 
does not. We also will give a direct proof of 
a particular case of Theorem~\ref{thm omega bdd by summable implies algebra}, which 
does not use the results of~\cite{FNikZarr98}. 

Given an analytic function $f(t) = \sum f_n t^n$, 
we denote by $[f]_N$ its truncated polynomial of degree $N$, that is,
\[
[f]_N := f_0 + f_1 t + \ldots + f_N t^N.
\]

\begin{exa}\label{exa alpha positive coeff}
Let $\si_2, \dots, \si_N$ be an arbitrary sequence of signs
(that is, a sequence of numbers $\pm 1$).
We assert that there are functions $\al, k$ meeting all the hypotheses
of Theorem~\ref{Muller-type-thm} such that $\operatorname{sign}(\al_n)=\si_n$,  
for $n=2, \dots, N$.
This is in contrast with Theorem~\ref{thmBCHMc}, where
the Nevanlinna-Pick condition was assumed:
$\al_n\le 0$ for $n\ge 2$.

To prove the existence of $\al$ and $k$ as above,
take a polynomial $\wt{\al}$ of degree $N$ such that $\wt{\al}_0 = 1, \wt{\al}_1 < 0$.
For $n=2,\dots, N$,
we set $\wt{\al}_{n}<0$ if $\si_n=-1$ and $\wt{\al}_{n}=0$
if $\si_n=1$.
Put $\wt{k} := [1/\wt{\al}]_N$.
The formula
\begin{equation}\label{eq Taylor coeff in terms of inverse}
\wt{k}_n = \sum_{\substack{s \ge 1 \\ n_1 + \cdots + n_s = n}}
(-1)^s \, \wt{\al}_{n_1} \cdots \wt{\al}_{n_s}
\end{equation}
shows that all the coefficients of $\wt{k}$ are positive.
We also require that neither $\wt{\al}$ nor the polynomial $\wt{k}$  vanish on $\ol{\D}$.
It is so if, for instance, $|\al_n|$ are sufficiently small for $n=2, \dots, N$.

Now perturb the coefficients $\wt{\al}_{j}$ that are equal to zero,
obtaining a new polynomial $\wh{\al}$ such that
\[
\wh{\al}_j :=
\left\{
\begin{array}{ll}
\ep & \tn{ if } \si_j=1\, \\
\wt{\al}_j & \tn{ otherwise}
\end{array}
\right.
\qquad
(2 \le j \le N).
\]
By continuity, if $\ep > 0$ is small enough, we can
guarantee that the polynomial $\wh{k} = [1/\wh{\al}]_N$ 
also has positive Taylor coefficients, 
and we can also guarantee that $\wh{k}$ 
(which is a slight perturbation of $\wt{k}$) 
does not vanish on $\ol{\D}$.

Finally, take as $k$ any function in $A_W$ with real Taylor coefficients such that
the first ones are
\[
k_0 = \wh{k}_0 = 1, \quad k_1 = \wh{k}_1, \quad \ldots , \quad k_{N} = \wh{k}_{N},
\]
and
\[
\dfrac{k_{n-j}}{k_n} \le C_0 \qquad (\forall n \ge 2j),
\]
for some constant $C_0$.
For instance, one can put $k_n = A n^{-b}$ for $n > N$,
with $A>0$ (small enough) and $b>1$.
Then $k \in A_W$ does not vanish on $\ol{\D}$.

Then obviously $k$ satisfies~\eqref{cond-Nik-Embry-fuerte}
and hence all the hypotheses of Theorem \ref{Muller-type-thm}. The function
$\al := 1/k$ in $A_W$ has the desired pattern of signs.

Finally, it is important to note that $\al_1 = -k_1$ is always negative.
\end{exa}

\begin{rem}
\label{exa alpha positive coeff 2}
It is also easy to see that whenever $\{k_n\}$ satisfies
\eqref{cond-Nik-Embry-fuerte}, any other sequence
$\{\tilde k_n\}$ with $k_0=1$ and
$c<\tilde k_n/k_n<C$ for $n > 1$, where $c, C$ are positive constants,
also satisfies this condition.
In particular, if $\{k_n\}$ satisfies
\eqref{cond-Nik-Embry-fuerte} and
$\{\tilde k_n\}$ is as above, where $C$ is sufficiently small,
then $k(t)$ is invertible in $A_W$, so that all
hypotheses of Theorem~\ref{Muller-type-thm}
are fulfilled.
So there are many examples of functions $k(t)$ meeting these hypotheses,
such that the quotients $k_n / k_{n+1}$ do not converge.
\end{rem}

Let us mention now some remarks on 
Theorem~\ref{thm omega bdd by summable implies algebra}.

\begin{rem}
\label{rem4_3}
It is immediate that the condition
\begin{equation}\label{eq condition on omega_n summable}
\dfrac{\om_n}{\om_j \om_{n-j}} \le \tau_j \quad (\forall n \ge 2j),  
\quad \tn{ where } \quad 
\sum_{j=0}^{\infty} \tau_j < \infty,
\end{equation}
implies \eqref{cond-Nik-ell-1} and \eqref{eq charact Banach algebra omega_n} 
(in particular, $\sup_n \om_{n+1}/\om_n<\infty$). 
Let us give a direct proof of Theorem~\ref{thm omega bdd by summable implies algebra} 
for this particular case.

Statement (i) follows using Proposition \ref{prop Aom Banach alg iff condition series}.

(ii) Put $g := 1/f$. Suppose that $g \not \in \ell^\infty(\om)$. This means that
\[
\sup_{n \ge 0} \, \abs{g_n} \om_n = \infty.
\]
Hence, it is clear that there exists a sequence $\{ \rho_n^0 \}$ in $[0,1]$ 
such that $\rho_n^0 \to 0$ (slowly) and
\begin{equation}\label{eq inclusion of rho_n 0}
\sup_{n \ge 0} \, \abs{g_n} \om_n \rho_n^0 = \infty.
\end{equation}

\begin{claims}
There exists a sequence $\{ \rho_n \}$ with
\begin{equation}\label{eq rho_n sandwich}
\rho_n^0 \le \rho_n \le 1 \quad \tn{ and } \quad \rho_n \to 0
\end{equation}
such that $\wt{\om}_n := \rho_n \om_n$ defines a Banach algebra $\ell^{\infty}(\wt{\om})$.
\end{claims}

Indeed, since $\sum \tau_j < \infty$, there exists a sequence of positive numbers
$\{ c_j \}$ such that $c_j \nearrow\infty$
and still $\sum c_j \tau_j < \infty$. Take any sequence $\{\rho_n\}$ that decreases,
tends to zero, and
satisfies $\rho_n\ge \max(\rho^0_n, 1/c_n)$. Then, for $\wt{\om}_n := \rho_n \om_n$ we have
\[
\dfrac{\wt{\om}_n}{\wt{\om}_j \wt{\om}_{n-j}} 
= \dfrac{\om_n}{\om_j \om_{n-j}} \dfrac{\rho_n}{\rho_j \rho_{n-j}} 
\le \dfrac{\om_n}{\om_j \om_{n-j}} \dfrac{1}{\rho_j} 
\le \tau_j c_j \qquad (\forall n \ge 2j).
\]
Since $ \sum \tau_j c_j < \infty$, 
Proposition~\ref{prop Aom Banach alg iff condition series} implies that
$\ell^{\infty}(\wt{\om})$ is a Banach algebra, and the proof of the claim is completed.

Now fix $\{\wt{\om}_n\}$ as in the claim. 
We may assume that $(\rho^0_n)^{1/n}\to 1$ and therefore $(\rho_n)^{1/n}\to 1$.
Since the polynomials are dense in the Banach algebra $\ell_0^\infty (\wt{\om})$, 
any complex homomorphism $\chi$
on $\ell_0^\infty (\wt{\om})$ is determined by its value on the power series $t$.
So the map $\chi\mapsto\chi(t)$ is injective and continuous from
the spectrum (the maximal ideal space) of  $\ell_0^\infty (\wt{\om})$ to $\C$.
Since $\wt{\om}_n^{1/n}\to 1$, its image contains $\D$ and is contained in $\overline{\D}$.
Hence the spectrum of $\ell_0^\infty (\wt{\om})$ is exactly the set
$\{\chi_\la: \la\in\overline{\D}\}$, where $\chi_\la(f)=f(\la)$.
(We borrow this argument from \cite{FNikZarr98}.) As
\[
f_n \wt{\om}_n = (f_n \om_n) \rho_n \to 0,
\]
we have $f \in \ell_0^\infty (\wt{\om})$. Then, using the Gelfand theory
(see, for instance, \cite[Chapter 10]{Rud1973}),
we get that $g = 1/f \in \ell_0^\infty (\wt{\om})$, which contradicts 
\eqref{eq inclusion of rho_n 0}. 
Therefore, the assumption $g  \notin \ell^\infty(\om)$ is false, 
as we wanted to prove.
\end{rem}

\begin{rem}
\label{rem4_4}
Notice that the above characterization of the spectrum of the algebra $\ell^\infty_0(\wt\om)$
(see the above Remark~\ref{rem4_3})
implies the following fact: the conditions~\eqref{eq charact Banach algebra omega_n} and
$\om_n^{1/n}\to 1$ imply that $\sum_n 1/\om_n<\infty$. 
This can be proved in an elementary way, without recurring to the Gelfand theory.

Indeed, by \eqref{eq charact Banach algebra omega_n}, there exists a constant $C>0$
such that
\[
\sum_{j=1}^{n} \dfrac{\om_n}{\om_j \om_{n-j}} \le C
\]
for every $n \ge 1$. Fix a positive integer $L$. Then obviously, for every $n \ge L$,
\begin{equation}
\label{eq sum quot om le C}
\sum_{j=1}^{L} \dfrac{\om_n}{\om_j \om_{n-j}} \le C.
\end{equation}
Let us see that
\begin{equation}\label{eq limsup min quotient omegas ge 1}
\limsup_{n \to \infty} \, \min_{1 \le j \le L} \,  \dfrac{\om_n}{\om_{n-j}} \ge 1.
\end{equation}
Indeed, if \eqref{eq limsup min quotient omegas ge 1} were false, then
there would exist some $r < 1$ and a positive integer $N$ such that
\[
\min_{1 \le j \le L} \,  \dfrac{\om_n}{\om_{n-j}} \le r \quad \text{for $n \ge N$.}
\]
From this, it is easy to see that
\[
\om_n \le r^{s_n} \, \left( \max_{0 \le k \le N} \,  \om_k \right),
\qquad
s_n := \left[ \dfrac{n-N}{L} \right] + 1,
\]
where $[a]$ denotes the integer part of $a$. Since $s_n$ behaves asymptotically as $n/L$,
it follows that $\limsup_{n \to \infty} \om_n^{1/n} \le r^{1/L} < 1$, 
which contradicts the hypothesis that $\om_n^{1/n} \to 1$. 
Therefore, \eqref{eq limsup min quotient omegas ge 1} is true.

Now, using \eqref{eq sum quot om le C}, it follows that
\[
C \ge \sum_{j=1}^{L} \dfrac{\om_n}{\om_j \om_{n-j}}
\ge
\left( \min_{1 \le j \le L} \, \dfrac{\om_n}{\om_{n-j}} \right) \sum_{j=1}^{L} \dfrac{1}{\om_j}.
\]
Taking $\limsup$ when $n \to \infty$, and using \eqref{eq limsup min quotient omegas ge 1},
we get that $\sum 1/ \om_j$ converges.
\end{rem}

The following statement shows that in the subcritical case,
the hypotheses of Theorem~\ref{Muller-type-thm} imply that the radius of convergence
of the series for $\al$ is equal to one.

\begin{prop}
If
$\lim k_n^{1/n}=1$ and $\al$ is of subcritical type,
then $\al$ does not continue analytically to any disc $R\D$,
where $R>1$.
\end{prop}

\begin{proof}
Since $k(t)$ has nonnegative Taylor coefficients, we have
$|k(t)|\le k(1)$ for all $t\in \D$. Using that $k=1/\al$, it follows that in the subcritical
case, $|\al(t)|\ge \al(1)>0$ for any $t\in \D$. So,
$\al$ cannot continue analytically to any disc $R\D$, where $R>1$, because in this case,
the radius of convergence of the Taylor series for $k$ would be greater than $1$.
\end{proof}

\section{Finite Defect}
\label{section finite defect}

It is well-known that in the classical Sz.-Nagy-Foias model,
the case of a finite rank (or Hilbert-Schmidt) defect
operator is an important one, where much more tools and results
are available.
In this section, we derive some consequences of our model theorems
for the case when an operator
$T\in\pweak{\al}$ is $\al$-modelable and the defect operator
$D=(\al(T^*, T))^{1/2}$ is of finite rank.

We will assume that the reproducing kernel Hilbert space $\CR_k$ is a Banach algebra
with respect to the multiplication of power series.
By~\cite[Proposition~32]{Shields-review}, it suffices to assume that
\[
\sup_n\sum_{j=0}^n \frac {k_j^2k_{n-j}^2}{k_n^2} <\infty;
\]
compare with the condition \eqref{eq charact Banach algebra omega_n}.
Put
\[
m_n= \inf_j \frac{k_j}{k_{n+j}} , \quad r_1=\lim_{n\to\infty} m_n^{1/n}. 
\]
This limit exists, see \cite[Proposition 12]{Shields-review}. 

We will assume that
\begin{equation}
\label{eq-r1}
r_1=\lim_{n\to\infty} k_n^{1/n}=1.
\end{equation}
Both equalities hold, in particular, if
$\lim k_{n+1}/k_n=1$.
The same is true if, for instance,
the last limit does not exist, but
$0<\sigma < k_n < C< \infty$ for all $n$ and
there is some $m \ge 2$ such that
$\lim_n k_{n+m}/k_n=1$.
We also are assuming here that the isometric part $S$ is not present in
the model of $T$. Hence, $T$ is unitarily equivalent to the restriction of
the backward shift $B_k\otimes I_\FD$ on the
space $\CH_k\otimes \FD$ to an invariant subspace $\Space$.
More generally, this applies to similarity instead of the
unitary equivalence 
(we bear in mind models of linear operators up to similarity, 
which are established in~\cite{ABY2}).

Here we prove the following result.

\begin{thm}
\label{thm-spectr-zero-set}
Suppose that $T$ is similar to a part of $B_k\otimes I_\FD$, acting on the space $\CH_k\otimes
\FD$, where $\CR_k$ is a Banach algebra and $\FD$ is finite dimensional. 
If the spectrum $\si(T)$ does not cover the
open disc $\D$, then $\si(T)\cap \D$ is contained in the zero set of a
non-zero function in~$\CR_k$.
\end{thm}

Let us start with some preliminary remarks.
Suppose that $T$ is as in the above Theorem~\ref{thm-spectr-zero-set}.
That is, $T$ is similar to $(B_k \otimes I_{\FD}) | \Space$,
where $\Space \ss \CH_k \otimes \FD$
is an invariant subspace of $B_k \otimes I_{\FD}$.
By fixing a basis in $\FD$, we may assume that $\FD=\C^d$, where $d=\dim \FD$. 
We will identify the space $\CH_k\otimes \FD$ with $\CH_k^d=\oplus_1^d \CH_k$, 
whose elements are columns with entries in $\CH_k$. 
The adjoint of $B_k$ on the space $\CH_k^d$ is the multiplication operator 
$M_z$ on the space $\CR_k^d$; 
this later space can be seen as a Banach module over the Banach algebra $\CR_k$. 
Put
\[
\CJ=\Space^\perp\subset \CR_k\otimes \FD.
\]
Then
$\CJ$ is $M_z$-invariant, and $T^*$ is similar to the quotient operator
\[
\CM_z : \CR_k^d / \CJ \to \CR_k^d / \CJ,
\qquad
\CM_z [f] = [zf].
\]
Here $[f]\in \CR_k^d / \CJ$ denotes the coset of a function $f$ in $\CR_k^d$.
We adapt some ideas from Richter's work~\cite{Richter87}, which
treated the case $d=1$.

\begin{defi}
(see ~\cite{Richter87}).
Let $\CJ$ be a subspace of $\CR_k^d$, invariant under $M_z$.

(1)
Given a point $\la\in\overline{\D}$, the space
\[
\CF(\la)=\CF_{\CJ}(\la) := \{ g(\la) \, : \, g \in \CJ \}
\]
will be referred to as \emph{the fiber of $\CJ$ over $\la$.} Note that $\CF(\la)$ is a
subspace of $\C^d$.

(2) By \emph{the spectrum of $\CJ$} we understand the set
\[
\si(\CJ) := \big\{ \la\in \overline{\D}:\quad  \CF_{\CJ}(\la)\ne \C^d\big\}.
\]
\end{defi}

It will be shown that Theorem~\ref{thm-spectr-zero-set} is an easy consequence
of the following result.

\begin{thm}
\label{thm-si-CL-si-Mz}
Given any subspace $\CJ$ of $\CR_k^d$, invariant under $M_z$, one has
$$\si(\CJ) \cap \D = \si(\CM_z) \cap \D.$$
\end{thm}

In the proof, we will use the following lemma

\begin{lemma}
\label{lem-division}
If
$g \in \CR_k$, $\la\in\D$ and $g(\la)=0$, then
$(z-\la)^{-1}g(z)\in \CR_k$.
\end{lemma}

\begin{proof}
Assume that $|\la|<r_1=1$.
By~\cite[Proposition~13]{Shields-review},
the operator $M_z-\la$ is bounded from below.
Notice that $(z-\la)^{-1}g(z)\in \CR_k$ if and only if $g$ belongs to the
closed subspace $\operatorname{Ran}(M_z-\la)$.
This happens if and only if $g$ is orthogonal to $\ker(M_z^*-\bar\la)$. 
This kernel is one-dimensional and is generated by the antilinear 
evaluation functional $g\mapsto \overline{g(\la)}$, 
which implies our assertion.
\end{proof}

\begin{proof}[Proof of Theorem~\ref{thm-si-CL-si-Mz}]
First we observe that $\eta\cdot \CJ\subset \CJ$ for any
$\eta$ in the algebra $\CR_k$, which is easy to get,
approximating $\eta$ by polynomials.

Assume first that $\la \in \D$, but $\la \notin \si(\CJ)$.
This means that $\CF(\la) = \C^d$.
Let us prove that $\la \not \in \si(\CM_z)$
(this will give the inclusion $\si(\CM_z) \cap \D \ss \si(\CJ) \cap \D$).
That is, we will see that $\CM_z-\la$ is invertible in $\CR_k^d / \CJ$.

\begin{claims}
If $h \in \CR_k^d$ and  $(z-\la)h \in \CJ$, then
$h\in\CJ$.
\end{claims}

Indeed, assume that $h$ satisfies these assumptions.
Since $\CF(\la) = \C^d$, there exist functions $\vp_1, \ldots , \vp_d$ in $\CJ$ such that
$\vp_j(\la) = e_j$ (where $\{ e_j \}$ is the canonical base in $\C^d$).
Consider the $d\times d$ matrix-valued function
\[
\Phi := (\vp_1 | \cdots | \vp_d) \in \CR_k^{d\times d},
\qquad \textnormal{ and set } \qquad
\vp := \det \Phi \in \CR_k.
\]
Note that $\Phi \ga = \ga_1 \vp_1 + \cdots + \ga_d \vp_d \in \CJ$ for every $\ga\in \CR_k^d$.
Hence,
$\vp h = \Phi \Phi^{\textnormal{ad}} h \in \CJ$.
We also observe that $(\vp - 1)h$ belongs to $\CJ$.
Indeed, since $\vp(\la) = 1$, we have
\[
(\vp - 1)h = \dfrac{\vp(z) - \vp(\la)}{z-\la} \,(z-\la) h \in \CJ,
\]
because $(\vp(z) - \vp(\la))/(z-\la) \in \CR_k$
by Lemma~\ref{lem-division} and $(z-\la) h \in \CJ$.
Therefore,
\[
h = \vp h - (\vp - 1)h \in \CJ,
\]
which proves our claim.

To check that $\CM_z-\la$ is invertible in $\CR_k^d / \CJ$,
take an arbitrary element $f$ in $\CR_k^d$,
and let us  study the solutions of the equation
\[
(\CM_z-\la) [h]= [f]
\]
with respect to an unknown coclass $[h]\in \CR_k^d / \CJ$.
By the above Claim, there is no more than one solution.
On the other hand, if we set
\[
h(z) = (z-\la)^{-1}\, \big(f(z)-\Phi(z)f(\la)\big),
\]
then by Lemma~\ref{lem-division}, $h\in \CR_k^d$, so that
$[h]$ is a solution of the above equation.
Note that $\Phi(\la) = I$.
It follows that the above formula defines a bounded map $[f]\mapsto [h]$, which
proves that the inverse to $\CM_z-\la$ exists and is bounded on $\CR_k^d / \CJ$.
This completes the proof of the inclusion
$\si(\CM_z)\cap \D\subset \si(\CJ) \cap \D$.

To prove the opposite inclusion, take a point $\la$ in $\si(\CJ) \cap \D$ and
let us see that $\la$ belongs to $\si(\CM_z)$.
In other words, we wish to prove that $\CM_z - \la$ is not invertible in $\CR_k^d / \CJ$.

Since $\la \in \si(\CJ)$, the fiber $\CF(\la)$ is not all $\C^d$.
Hence, there exists a nonzero antilinear functional $\Psi$
on $\C^d$ such that $\Psi | \CF(\la) \equiv 0$.
It defines an  antilinear functional on $\CR_k^d$, given by
\[
\widehat{\Psi}(f) = \Psi(f(\la)).
\]
Note that $\widehat{\Psi} \ne 0$, but $\widehat{\Psi} | \CJ \equiv 0$.
Hence, we obtain the antilinear functional $\widetilde{\Psi}$
on the quotient $\CR_k^d / \CJ$, given by
\[
\widetilde{\Psi} : \CR_k^d / \CJ \to \C,
\qquad
\widetilde{\Psi}([f]) := \widehat{\Psi}(f).
\]
For every $f \in \CR_k^d$ we have
\[
\left\langle  (\CM_z - \la)^*\widetilde{\Psi}, [f] \right\rangle
=
\widetilde{\Psi}((\CM_z - \la)[f])
=
\widehat{\Psi}((z - \la)f) = 0,
\]
because $(z - \la)f(z)$ vanishes for $z=\la$.
Hence, $(\CM_z - \la)^* \widetilde{\Psi} = 0$.
Since $\widetilde{\Psi} \ne 0$, we get
that $(\CM_z - \la)$ is not invertible in $\CR_k^d / \CJ$,
as we wanted to prove.
\end{proof}

We remark that the above Claim is very close 
to Corollary~3.8 in the Richter's paper~\cite{Richter87}, which 
can be stated as follows: for any (reasonable) Banach algebra 
of analytic functions on $\D$, continuable to $\ol\D$, 
any its invariant subspace 
(which is the same as an ideal) 	
has index one. Richter only studies scalar-valued 
algebras $\CR_k$, and the above Claim can be seen 
as an extension of Richter's result to the case 
of a vector-valued space $\CR_k\otimes \CE$, where 
$\dim \CE<\infty$. The indices of invariant subspaces 
of vector-valued spaces of analytic functions have been 
studied by Carlsson in~\cite{Carlsson09}. 
In \cite[Section~10]{AlemHedmlmRicht2005}, 
one can find a review of the index phenomena, 
related to invariant subspaces of Bergman spaces. 

\begin{proof}[Proof of Theorem \ref{thm-spectr-zero-set}]
We conserve the notation of the above proof.
Fix any point $\la$ in $\D\setminus \si(T)$.
Then, as above, there exist functions
$\vp_j\in \CJ$ such that
$\vp_j(\la)=e_j$, $j=1,\dots d$. Define the $d\times d$ matrix function $\Phi(z)$
as above and put $\vp(z)=\det \Phi(z)$, then $\vp$ belongs to $\CR_k$.
Observe that $\vp \not \equiv 0$.
Notice that the fiber $\CF_\CJ(z)$ equals to $\C^d$ whenever
$\Phi(z)$ is invertible, that is, whenever $\vp(z)\ne 0$.
Hence $\si(\CJ)$ is contained in the zero set of $\vp$.
By Theorem~\ref{thm-si-CL-si-Mz},
$\si(T)\cap\D=\si(\CJ)\cap\D$,
and this implies the statement of the theorem.
\end{proof}

For the proof of Theorem~\ref{thm-conseq-Carleson} we need the following lemma.

\begin{lemma}
\label{lem-conseq-Carleson}
Assume the hypothesis of Theorem~\ref{thm-spectr-zero-set}
and also \eqref{eq sum kn bound C epsilon}. Then, there exists
a positive number $s$ such that
the functions in $\CR_k$
are H\"{o}lder continuous of order $s$ on $\ol{\D}$.
\end{lemma}

\begin{proof}
It is easy to see that \eqref{eq sum kn bound C epsilon} implies that
for a sufficiently small $s\in (0,1)$,
\begin{equation}\label{eq-for-Lips-cond}
\sup_{0<t<1} t^{-s} \left(\sum_{t^{s-1}}^{\infty} k_n\right)^{1/2} < \infty.
\end{equation}
Fix such $s$, and let $f$ be a function in $\CR_k$.
To prove that $f$ is H\"{o}lder continuous of order $s$ on $\ol{\D}$,
it is enough to show that $f$ is H\"{o}lder continuous of order $s$ in $\T$
(see \cite{HL32}).
By a rotation argument, we just need to prove that
\[
\sup_{\substack{\theta \neq 0 \\  \theta \in [-\pi, \pi]}}
\abs{\theta}^{-s}  \abs{f(1) - f(e^{i\theta})} < \infty.
\]
By the Cauchy-Schwarz inequality, it is enough to prove that
\begin{equation}
\label{eq-conseq-Carl-2}
\sup_{\substack{\theta \neq 0 \\  \theta \in [-\pi,\pi]}}
\abs{\theta}^{-s}
\left(
\sum_{n=0}^{\infty} k_n \abs{1-e^{in\theta}}^2
\right)^{1/2}
< \infty.
\end{equation}
Note that $\abs{1-e^{in\theta}}^2 \le n^2 \theta^2$,
hence $\abs{1-e^{in\theta}}^2 \le \abs{\theta}^{2s}$ if $n \le \abs{\theta}^{s-1}$.
Therefore
\[
\abs{\theta}^{-s}
\left(
\sum_{n=0}^{\infty} k_n \abs{1-e^{in\theta}}^2
\right)^{1/2}
\le
\left(
\sum_{n=0}^{\abs{\theta}^{s-1}} k_n
\right)^{1/2}
+ 2 \abs{\theta}^{-s}
\left(
\sum_{\abs{\theta}^{s-1}}^{\infty} k_n
\right)^{1/2},
\]
which is uniformly bounded because 
$\sum k_n$ converges and \eqref{eq-for-Lips-cond}
holds. Hence \eqref{eq-conseq-Carl-2} follows and the statement is proved.
\end{proof}

\begin{proof}[Proof of Theorem~\ref{thm-conseq-Carleson}]
By Theorem~\ref{thm-spectr-zero-set}, we have that
$\si(T) \cap \D$ is contained in the zero set of a
non-zero function $f$ in $\CR_k$.
By Lemma~\ref{lem-conseq-Carleson}, $f$ is H\"{o}lder continuous of order $s$ for
some $s>0$.
Note that $E$ is a set of uniqueness for $f$. Hence the statement follows
using \cite[Theorem~1]{Car52}.
\end{proof}

We conjecture that the statements of
Theorems~\ref{thm-conseq-Carleson} and~\ref{thm-spectr-zero-set} are valid
for the whole spectrum $\si(T)$. Some of our arguments do not apply and should
be changed in order to prove it. 

\section{Ergodic properties of $a$-contractions}
\label{section Consequences for ergodic theory}

In this section we focus only on functions $\al$ of the form
\begin{equation}\label{eq alpha power a}
\al(t) = (1-t)^a
\end{equation}
for some $a>0$.
Recall that if $(1-t)^a(T^*,T) \ge 0$ for some $T \in L(H)$,
then we say that $T$ is an $a$-contraction.
Now
\begin{equation}
\label{k-a}
k(t)= (1-t)^{-a} = \sum_{n=0}^{\infty} k^a(n) t^n \qquad (\abs{t}<1).
\end{equation}

Observe that $k^a(n)>0$ for all $n\ge 0$.
It follows that  $\al(t)=(1-t)^a$ satisfies Hypotheses~\ref{hypo-al}.
Moreover, here we are in the critical case.

If $0<a\le 1$, then $\al_n\le 0$ for $n>0$, so that in this case we can apply
Theorem~\ref{thmBCHMc}
to obtain a model for $a$-contractions.
This was singled out as an important particular case in~\cite{CH18}.
With the help of this model, 
we will derive here some ergodic properties of $a$-contractions for these values
of $a$. We refer to the book \cite{GP11} for a treatment of ergodic theory in the context
of the theory of linear operators.

Notice that
\[
k^a(n) = (-1)^n {a \choose n} =
\begin{cases}
\dfrac{a(a+1)\cdots (a+n-1)}{n!}  & \tn{ if } n \ge 1\\
1 & \tn{ if } n = 0
\end{cases}.
\]
These numbers are called \emph{Ces\`{a}ro numbers}.
See \cite[Volume I, p. 77]{Zygmund}.
We will need the following well-known facts about
their asymptotic behavior.

\begin{prop}\label{prop properties of Cesaro numbers}
If $a\in \C\setminus\{0,-1,-2,\dots\}$, then
\[
k^{a}(n)=\frac{\Gamma(n+a)}{\Gamma(a)\Gamma(n+1)}
={n+a-1\choose a-1}\qquad \forall n \ge 0,
\]
where $\Gamma$ is Euler's Gamma function. Therefore
\begin{equation}\label{asymp}
k^{a}(n)=\frac{n^{a-1}}{\Gamma(a)}(1+O(1/n)) \qquad \tn{ as } n\to\infty.
\end{equation}
Moreover, if $0<a \le 1$, then
\[
\frac{(n+1)^{a -1}}{\Gamma(a)}\leq k^{a}(n)
\leq \frac{n^{a -1}}{\Gamma(a)} \qquad \forall n \ge 1.
\]
\end{prop}

\begin{proof}
See \cite[Volume I, p. 77, Equation (1.18)]{Zygmund} and \cite[Equation (1)]{ET}. 
The last inequality follows from the Gautschi inequality (see \cite[Equation (7)]{Gau59}).
\end{proof}

Any contraction $T$ on $H$ is an $a$-contraction 
for any $a\in (0,1)$. Indeed, in this case, $\al_n\le 0$ for all $n\ge 1$. 
Hence for any $x\in H$, 
$\sum_{n\ge 1} \al_n \|T^nx\|^2\ge (\sum_{n\ge 1} \al_n)\|x\|^2=-\|x\|^2$, 
which implies that $\al(T^*, T)\ge 0$.

Recall that in order to emphasize the dependence on the exponent
$a$ in \eqref{eq alpha power a},
we denote the space $\CH_k$ by $\CH_a$
(see Notation~\ref{Notation Bs}), and
use the notation $B_a$ and $F_a$ (these two operators
act on $\CH_a$).
In the same way,
when $T \in \pweak{\al}$ (that is, when $T$ is an $a$-contraction),
we will write $T \in \pweak{a}$,
and instead of $\weak{\al}$ we will use the notation $\weak{a}$.

The \emph{weighted space of Bergman-Dirichlet type} $\mathcal{D}_a$, where $a$
is a real parameter,
consists of all the analytic functions $f$ in $\D$ with finite norm
\[
\norm{f}_{\mathcal{D}_s}
:=
\left(
\sum_{n=0}^{\infty} (n+1)^a \abs{f_n}^2
\right)^{1/2}.
\]
In fact, it is a Bergman-type space if $a<0$, and
is a Dirichlet-type space if $a>0$. For $a=0$, we get the Hardy space.

Theorem~\ref{thmBCHMc} yields that for $0<a\le 1$, any $a$-contraction is modelable
as a part of an operator $(B_a\otimes I_\CE)\oplus S$, where $S$ is an isometry.
Recall that the adjoint to the operator $B_a=B_k$ on
$\CH_k$ is the operator $M_zg(z)=zg(z)$ on $\CR_k$, which is a space
with the weighted norm
\[
\|g\|^2_{\CR_k}=\sum_{n=0}^\infty k_n^{-1}|g_n|^2.
\]
Hence the characterization of invariant subspaces of $M_z$ on $\CR_k$
becomes important. In many cases, this question is related to the description of what
is called inner functions in $\CR_k$.
Since $k_n \asymp (n+1)^{a-1}$, 
the norm in $\CR_k$ is equivalent to the norm in $\CD_{-a+1}$, 
which is a Dirichlet-type space.
One can find results in this direction in the thesis of 
Schillo \cite{Schillo-tesis}, 
in \cite{PauPelaez2011} by Pau and Pel\'{a}ez,
and in the papers~\cite{Seco-Dirichlet-inner-2019}
and~\cite{BeneteauSeco18} by Seco and coauthors; see also references therein.

As a consequence of Theorem~\ref{thm characterization backward in Calpha}
we obtain the following result.

\begin{thm}
\label{thm B_s a-contraction}
Let $a$ and $s$ be positive numbers. Then the following is true.
\begin{itemize}
\item[\tn{(i)}] 
$B_s \in \weak{a}$.
\item[\tn{(ii)}] 
$B_s$ is an $a$-contraction if and only if $a \le s$.
\end{itemize}
\end{thm}

\begin{proof}
Using the notation of Theorem~\ref{thm characterization backward in Calpha}
we have $\ka(t) = (1-t)^{-s}$ and $\al(t) = (1-t)^a$.
Hence $\be(t)=p(t) \pm \al(t)$, where $p(t)$ is a polynomial, say
$p(t) = p_0 + p_1 t + \cdots + p_N t^N$, with all the coefficients $p_j$ positive.
Then
\[
\ga(t) = (p(t) \pm \al(t)) \ka(t) = p(t) \ka(t) \pm (1-t)^{a-s} =: \wt{\ga}(t) + \wh{\ga}(t).
\]
To prove (i), it is enough to show that
\[
\sup_{m\ge 0} \dfrac{\abs{\wt{\ga}_m}}{\ka_m} < \infty,
\quad \tn{ and } \quad
\sup_{m\ge 0} \dfrac{\abs{\wh{\ga}_m}}{\ka_m} < \infty.
\]
On one hand, for $m\ge N$ we have
\[
\abs{\wt{\ga}_m} \le (p_0 + \cdots + p_N) \cdot \max\{ \ka_m, \ldots , \ka_{m-N} \}
\lesssim \ka_m.
\]
On the other hand,
\[
\abs{\wh{\ga}_m} \asymp \dfrac{1}{m^{a-s+1}} \le \dfrac{1}{m^{-s+1}} \asymp \ka_m.
\]
Therefore (i) follows.
Observe that (ii) is an immediate consequence of
Theorem~\ref{thm characterization backward in Calpha}~(ii),
since $\al(t)\ka(t) = (1-t)^{a-s}$.
\end{proof}

Note that it is immediate that
\begin{equation}\label{eq norm Bs}
\norm{B_s^m}^2
=
\sup_{n \ge 0} \dfrac{\ka_n}{\ka_{n+m}}
= \begin{cases}
1 & \text{if } 1 \le s \\
1/\ka_m & \text{if } 0 < s < 1,
\end{cases}
\end{equation}
and
\begin{equation}\label{eq norm Fs}
\norm{F_s^m}^2
=
\sup_{n \ge 0} \dfrac{\ka_{n+m}}{\ka_n}
= \begin{cases}
\ka_m & \text{if } 1 \le s \\
1 & \text{if } 0 < s < 1,
\end{cases}
\end{equation}
for every $m \ge 0$.
Therefore
\begin{equation}\label{eq asymp Bs and Fs}
\norm{B_s^m}^2 \asymp (m+1)^{\max \{ 1-s, 0 \}}
\quad \tn{ and } \quad
\norm{F_s^m}^2 \asymp (m+1)^{\max \{ s-1, 0 \}}.
\end{equation}

As an easy consequence we obtain the following example,
which shows two relevant facts: 
1) there are $a$-contractions that are not similar to contractions, and
2) the importance of considering the strong operator topology in
the convergence of $\sum \al_n T^{*n}T^n$.

\begin{exa}
\label{exa B_s a-contr not similar contr}
Taking $a=s\in (0,1)$ in Theorem~\ref{thm B_s a-contraction}, we get that
$B_{a}$ is an $a$-contraction. It is not similar to a contraction,
since it is not power bounded. Moreover, if $a=s\le 1/2$, then 
\[
\sum_{n=0}^{\infty} \abs{\al_n} \, \| {B_{a}^n} \|^2
\asymp
\sum_{n=0}^{\infty} (n+1)^{-1-a} \,  (n+1)^{1-a} = \sum_{n=0}^{\infty} (n+1)^{-2a} = \infty,
\]
and therefore the series $\sum \al_n B_{a}^{*n} B_{a}^n$ does not converge in
the uniform operator topology in $L(H)$.
Note that, obviously, the model of the $a$-contraction $B_{a}$ is itself.
\end{exa}

Let us study now some ergodic properties of $a$-contractions.

\begin{defi}\label{defi Cesaro means and Cesaro bounded}
Let $a \ge 0$. 
For any bounded linear operator $T$ on a Banach space $X$, 
we call the operators $\{ M^a_T(n) \}_{n \ge 0}$ given by
\[
M^a_T(n) := \dfrac{1}{k^{a+1}(n)} \sum_{j=0}^{n} k^a(n-j) T^j,
\]
the \emph{Ces\`aro means of order $a$ of $T$}. 
When this family of operators is uniformly bounded, that is,
\[
\sup_{n \ge 0} \norm{M^a_T(n)} < \infty,
\]
we say that $T$ is \emph{$(C,a)$-bounded}.
\end{defi}

\begin{rems}
\quad
\begin{enumerate}[\rm (i)]
\item
Note that $\sum_{j=0}^{n} k^a(j) = k^{a+1}(n)$ for any $a\ge 0$.
Also, if $a\ge 0$, then $k^a(j) \ge 0$ for every $j \ge 0$.
\item
If $a=0$, then $M_T^0(n) = T^n$. Hence $(C,0)$-boundedness is just power boundedness.
\item
If $a=1$, then $M_T^1(n) = (n+1)^{-1} \sum_{j=0}^{n} T^j$.
Hence $(C,1)$-boundedness is just Ces\`aro boundedness.
\item
It is well-known that if $0 \le a < b$, 
then $(C,a)$-boundedness implies $(C,b)$-boundedness.
The converse is not true in general. For example,
the Assani matrix
\[
T
=
\begin{pmatrix}
-1 & 2 \\
0 & -1
\end{pmatrix}
\]
is $(C,1)$-bounded, but since
\[
T^n
=
\begin{pmatrix}
(-1)^n & (-1)^{n+1}2n \\
0 & (-1)^n
\end{pmatrix}
\]
it is not power bounded (see \cite[Section 4.7]{Emi85}).
\end{enumerate}
\end{rems}

\begin{defi}
If the sequence of operators $\{ M^a_T(n) \}_{n \ge 0}$ given in 
Definition \ref{defi Cesaro means and Cesaro bounded}
converges in the strong operator topology, 
we say that $T$ is \emph{$(C,a)$-mean ergodic}.

If $T$ is $(C,1)$-mean ergodic, it is conventional
just to say that $T$ is \emph{mean ergodic}.
\end{defi}

There is a well established literature on $(C,a)$-bounded operators,
which explores quite a number of properties and their interplays.
Properties, characterization through functional calculus and 
ergodic results for $(C,a)$-bounded operators 
can be found in \cite{ALMV16, AS16, De00, Eddari1, Eddari2, Emi85, LSAS}
and references therein.
The connection of these operators and ergodicity dates back to the fourties of last century,
see \cite{Coh40} and \cite{Hil45}.
In the latter paper, E. Hille studies $(C,a)$-mean ergodicity in terms of Abel convergence
(that is, via the resolvent operator).
As application, the well known mean ergodic von Neumann's theorem for unitary groups
on Hilbert spaces is extended to $(C,a)$-mean ergodicity 
for every $a>0$ \cite[p. 255]{Hil45}.
Also, the $(C,a)$-ergodicity on $L_1(0,1)$ of fractional (Riemann-Liouville) integrals
is elucidated in \cite[Theorem 11]{Hil45}.
In particular, if $V$ is the Volterra operator then $T_V:=I-V$, as operator on $L_1(0,1)$,
is not power-bounded, and it is $(C,a)$-mean ergodic if and only if 
$a>1/2$ \cite[Theorem 11]{Hil45}.
This result can be extended to $T_V$ acting on $L_p(0,1)$, $1<p<\infty$,
using estimates given in \cite{MSZ05}, see \cite[Section 10]{AGL19}.

In \cite{LH15}, Luo and Hou introduced a new notion of boundedness: 
a bounded linear operator $T$ on a Banach space $X$ is
said to be \emph{absolutely Ces\`aro bounded} if
\[
\sup_{n \ge 0} \dfrac{1}{n+1} \sum_{j=0}^{n} \norm{T^j x} \lesssim \norm{x}
\]
for every $x \in X$.
In \cite{BBMP20}, the authors study the ergodic behaviour for this class of operators.
The above definition has been
extended recently by Abadias and Bonilla in \cite{AB19}:
$T$ is said to be \emph{absolutely $(C,a)$-Ces\`aro bounded}
for some $a>0$ if
\[
\sup_{n \ge 0} \dfrac{1}{k^{a+1}(n)} \sum_{j=0}^{n} k^a (n-j) \norm{T^j x} 
\lesssim \norm{x}
\]
for every $x \in X$. Note that for $a=1$  the definition of Luo and Hou is recovered.

\begin{rem}
It is well-known that the following implications hold:
\[
\begin{split}
&\text{Power bounded }\Rightarrow\text{ Absolutely }(C, a)\text{-bounded } \\
&\Rightarrow\ (C, a)\text{-bounded  }\Rightarrow\ \|T^n\|=O(n^a).
\end{split}
\]
The first two implications are straightforward.
For the sake of completeness, we give a proof of the last one.
Suppose $T$ is $(C,a)$-bounded for some $a \ge 0$.
We denote by $[a]$ the integer part of $a$.
Then, for $n>[a]$, we have
\[
\begin{split}
\norm{T^n}
&=
\bigg\|
\sum_{j=0}^n k^{-a}(j) \sum_{m=0}^{n-j} k^{a}(n-j-m)T^m
\bigg\| \\
&\lesssim
\sum_{j=0}^{n} \abs{k^{-a}(j)} k^{a+1}(n-j) \\
&=
\sum_{j=0}^{[a]}(-1)^j k^{-a}(j)k^{a+1}(n-j)
+\sum_{j=[a]+1}^n (-1)^{[a]+1}k^{-a}(j)k^{a+1}(n-j) \\
&=
\sum_{j=0}^{[a]}\left((-1)^j+(-1)^{[a]}\right) k^{-a}(j)k^{a+1}(n-j)
+(-1)^{[a]+1}  \sum_{j=0}^{n} k^{-a}(j)k^{a+1}(n-j) \\
&\lesssim
\sum_{j=0}^{[a]} \abs{k^{-a}(j)} k^{a+1}(n-j) + k^1(n)
\lesssim
k^{a+1}(n)
\asymp
(n+1)^a\, .
\end{split}
\]
\end{rem}

The following extension of the above definitions will be important for us.
\begin{defi}\label{defi abs Cap bdd}
Let $a>0$ and $p\geq 1$. 
We say that a bounded linear operator $T$ on a Banach space $X$ is
\emph{$(C, a, p)$-bounded} if
\[
\sup_{n \ge 0} \frac{1}{k^{a+1}(n)}\sum _{j=0}^n  k^{a}(n-j) \|T^jx\|^p  \lesssim \|x\|^p,
\]
for all $x\in X$.
\end{defi}
Note that for $p=1$ this definition is just the absolute $(C,a)$-boundedness.
The case $a=1$ has been recently considered in \cite{CCEL20}.
We will use
the term \emph{quadratically $(C,a)$-bounded} instead of $(C,a,2)$-bounded.

Using the asymptotics $k^a(n) \asymp (n+1)^{a-1}$ given in \eqref{asymp}, it is easy to see
that $T$ is $(C,a,p)$-bounded if and only if
\begin{equation}\label{eq char of Cap-bdd}
\sup_{n \ge 0} \frac{1}{(n+1)^a}\sum _{j=0}^n  (n+1-j)^{a-1} \norm{T^j x}^p  \lesssim \norm{x}^p
\qquad (\forall x \in X).
\end{equation}

The  following observation
will be essential for the proof of Theorem \ref{thm Ca to Cesaro bounded}.
\begin{lemma}\label{lemma inherit by a part a tensor IR}
The following holds.
\begin{itemize}
\item[\tn{(i)}]
If $T$ is $(C, a, p)$-bounded, then any part of $T$ is also $(C, a, p)$-bounded.
\item[\tn{(ii)}]
If $T_1$ and $T_2$ are $(C, a, p)$-bounded,
then any direct sum $T_1 \dotplus T_2$ is also $(C, a, p)$-bounded.
\item[\tn{(iii)}]
Let $T$ be a bounded linear operator on a Hilbert space. 
If $T$ is quadratically $(C, a)$-bounded,
then $T\otimes I_{\CE}$ is also quadratically $(C, a)$-bounded,
where $I_\CE$ is the identity operator on some Hilbert space $\CE$.
\end{itemize}
\end{lemma}

\begin{proof}
(i) and (ii) are immediate.
For (iii) note that if $d=\dim \CE \le \infty$, 
then the orthogonal sum of $d$ copies of $T$ is clearly
quadratically $(C, a)$-bounded (by the Pythagoras Theorem).
\end{proof}

The following result is very useful. Its proof is simple, and we omit it.
\begin{lemma}\label{lemma Cap bdd implies Cbp bdd}
Let $0 \le a < b$.
Then $(C,a,p)$-boundedness implies $(C,b,p)$-boundedness.
\end{lemma}

This lemma shows an inclusion of classes of operators.
By \cite[Corollaries 2.2 and 2.3]{AB19},  if $T$ is
$(C,a,1)$-bounded then $\|T^n\|=o(n^{a})$ for $0<a\leq 1$ and $\|T^n\|=O(n)$ for $a> 1$.
The following result explains
why the case $a=1$ is special.

\begin{thm}\label{thm Cap equals C1p}
If $a>1$ and $p\ge 1$, then $(C,a,p)$-boundedness is equivalent to $(C,1,p)$-boundedness.
\end{thm}
\begin{proof}
Fix $a>1$ and $p\ge 1$. 
By the above Lemma, we only need to prove that any $(C,a,p)$-bounded operator $T$ is
$(C,1,p)$-bounded.
Let $T$ is $(C,a,p)$-bounded. Then
\begin{equation}
\label{ineq1}
\frac{1}{k^{a+1}(2n)}\sum _{j=0}^{2n}  k^{a}(2n-j)
\norm{T^j x}^p  \lesssim \norm{x}^p,
\end{equation}
for every $n \ge 0$, and every $x \in X$.
Since $a>1$, $k^a(m)$ is
an increasing function of $m$.
In particular, $k^a(n)\leq k^a(2n-j)$ for $j=0,\ldots, n$. Hence
\begin{equation}
\label{ineq2}
k^{a}(n)\sum_{j=0}^{n} \norm{T^jx}^p \le \sum _{j=0}^{2n}  k^{a}(2n-j) \norm{T^jx}^p,
\end{equation}
By \eqref{ineq2} and \eqref{ineq1},
\[
\sum_{j=0}^{n} \norm{T^jx}^p \lesssim \frac{k^{a+1}(2n)}{k^a(n)} \norm{x}^p
\lesssim (n+1) \norm{x}^p,
\]
which means that $T$ is $(C,1,p)$-bounded.
\end{proof}

\begin{thm}\label{th4.4}
Let $a>0$ and $1\leq q<p$. 
If $T$ is $(C, a, p)$-bounded, then it is also $(C, b, q)$-bounded for each $b>qa /p$.
In particular, $(C, a, p)$-boundedness implies $(C, a, q)$-boundedness.
\end{thm}

\begin{proof}
Let us first recall that if $r > -1$, then
\begin{equation}\label{eq sum estimate by integral}
\sum_{j=1}^{m} j^r \, \lesssim \, m^{r + 1} \qquad (\forall m \ge 1).
\end{equation}
Let $T$ be $(C, a, p)$-bounded and let $b>qa /p$.
Suppose first that $b \neq 1$, and put
\[
s:=\frac{p}{p-q}, \qquad s':=\frac{p}{q}, \qquad \gamma:=\frac{q(a-1)}{p(b-1)}.
\]
Note that $s$ and $s'$ are positive and satisfy $1/s + 1/s' = 1$.
Since
\[
(b-1)(1-\ga)s = \dfrac{pb-qa}{p-q} - 1 > -1
\qquad
\tn{ and }
\qquad
(b-1)\ga s' = a-1,
\]
using H\"older's inequality and \eqref{eq sum estimate by integral} it follows that
\[
\begin{split}
&\dfrac{1}{(n+1)^b} \sum_{j=0}^{n} (n+1-j)^{b-1} \norm{T^j x}^q \\
&\le
\dfrac{1}{(n+1)^b}
\left( \sum_{j=0}^{n} (n+1-j)^{(b-1)(1-\ga)s} \right)^{1/s}
\left( \sum_{j=0}^{n} (n+1-j)^{(b-1)\ga s'} \norm{T^j x}^{qs'} \right)^{1/s'} \\
&\lesssim
(n+1)^{-qa/p}
 \left( \sum_{j=0}^{n} (n+1-j)^{a-1} \norm{T^j x}^{p} \right)^{q/p} \\
&= \left( \dfrac{1}{(n+1)^a} \sum_{j=0}^{n} (n+1-j)^{a-1} \norm{T^j x}^{p} \right)^{q/p}
\end{split}
\]
for every $x \in X$ and every non-negative integer $n$.
Hence the statement follows using \eqref{eq char of Cap-bdd}.

Now suppose that $b=1$. Take any $b' \in (qa/p,1)$.
We have already proved that $T$ is $(C, b', p)$-bounded. 
Then, by Lemma~\ref{lemma Cap bdd implies Cbp bdd}, it
follows that $T$ is $(C, 1, p)$-bounded. This completes the proof.
\end{proof}
\begin{lemma}\label{lemma S is Cap bounded}
Let $a>0$ and $p \ge 1$. Then every isometry $S$ is $(C,a,p)$-bounded.
\end{lemma}
\begin{proof}
This is immediate, since indeed
\begin{equation}\label{eq S is Cap bounded}
\frac{1}{k^{a+1}(n)}\sum _{j=0}^n  k^{a}(n-j) \|S^jx\|^p
=
\frac{1}{k^{a+1}(n)} \left( \sum _{j=0}^n  k^{a}(n-j) \right) \|x\|^p
=
 \|x\|^p
\end{equation}
for every $x \in X$.
\end{proof}

\begin{lemma}\label{quadraticShift}
Let $0<s<1$ and let $a>0$. 
Then $B_s$ is quadratically $(C,a)$-bounded if and only if $1-s<a$.
Moreover, for $1-s<a$ we have
\begin{equation}\label{eq lim Bs Ca bdd is 0}
\lim_{n\to\infty}\frac{1}{k^{a+1}(n)}\displaystyle\sum_{j=0}^n k^{a}(n-j)\|B_s^jx\|^2=0
\qquad \qquad
(\forall x \in \CH_s).
\end{equation}
\end{lemma}

\begin{proof}
Recall the notation $e_n=t^n\in \CH_k = \CH_s$, where $k(t)=(1-t)^{-s}$.
Suppose that $a=1-s$. Then
\begin{equation}\label{eq quadraticShift for 1-s}
\begin{split}
\dfrac{1}{(n+1)^a} \sum_{j=0}^{n} (n+1-j)^{a-1} \norm{B_s^j e_n}^2
& \gtrsim
\dfrac{1}{(n+1)^{1-s}} \sum_{j=0}^{n} (n+1-j)^{-s} (n+1-j)^{s-1} \\
&=
\dfrac{1}{(n+1)^{1-s}} \sum_{j=1}^{n+1} j^{-1}
\, \gtrsim \,
\log(n+2) \norm{e_n}^2
\end{split}
\end{equation}
for every $n$. Therefore $B_s$ is not quadratically $(C,1-s)$-bounded, and
by Lemma \ref{lemma Cap bdd implies Cbp bdd} we obtain that 
$B_s$ is not quadratically $(C,a)$-bounded for $a < 1-s$.

Let us assume now that $1-s<a \le 1$ and fix $x \in \CH_s$. Write $x$ in the form
$x = \sum x_m e_m$, where $x_m \in \C$. Then
\[
\norm{B_s^j x}^2 = \sum_{m=j}^{\infty} k^s(m-j) \abs{x_m}^2
\, \lesssim \,  \sum_{m=j}^{\infty} (m+1-j)^{s-1} \abs{x_m}^2,
\]
for every $j \ge 0$.
Hence
\begingroup
\allowdisplaybreaks
\begin{align*} 
\dfrac{1}{(n+1)^a}  \sum_{j=0}^{n} (n+1-j)^{a-1} & \norm{B_s^j x}^2 \\
&\lesssim
\dfrac{1}{(n+1)^a} \sum_{j=0}^{n} (n+1-j)^{a-1} \sum_{m=j}^{\infty} (m+1-j)^{s-1} \abs{x_m}^2 \\
&= \dfrac{1}{(n+1)^a} \sum_{m=0}^{n} \abs{x_m}^2 \sum_{j=0}^{m} (n+1-j)^{a-1} (m+1-j)^{s-1} \\
&+ \dfrac{1}{(n+1)^a}\sum_{m=n+1}^{2n} \abs{x_m}^2 \sum_{j=0}^{n} (n+1-j)^{a-1} (m+1-j)^{s-1} \\
&+ \dfrac{1}{(n+1)^a}\sum_{m=2n+1}^{\infty} \abs{x_m}^2 \sum_{j=0}^{n} (n+1-j)^{a-1} (m+1-j)^{s-1} \\
&=:
(I) + (II) + (III).
\end{align*}
\endgroup
In (I), note that since $1-s<a \le 1$, and $m \le n$, we have
\begin{equation}\label{eq asymp m le n}
\sum_{j=0}^{m} (n+1-j)^{a-1} (m+1-j)^{s-1}
\le
\sum_{j=0}^{m+1} (m+1-j)^{a+s-2}
\lesssim (m+1)^{a+s-1},
\end{equation}
where in the last estimate we used \eqref{eq sum estimate by integral}.
Therefore
\[
\begin{split}
(I)
&\lesssim
\dfrac{1}{(n+1)^a} \sum_{m=0}^{n} \abs{x_m}^2 (m+1)^{a+s-1}
=
\left\{ \sum_{m=0}^{[\sqrt{n}]} 
+ \sum_{m= [\sqrt{n}] + 1}^{n} \right\} \abs{x_m}^2 \,  \dfrac{(m+1)^{a+s-1}}{(n+1)^a} \\
&\lesssim
\dfrac{\norm{x}^2}{\sqrt{n^a}} + \sum_{m= [\sqrt{n}] + 1}^{n} \abs{x_m}^2 (m+1)^{s-1}
\longrightarrow 0 \quad (\tn{as } n \to \infty).
\end{split}
\]
In (II), using that $m>n$ and $s-1<0$, we have
\[
\sum_{j=0}^{n} (n+1-j)^{a-1} (m+1-j)^{s-1} \le \sum_{j=0}^{n} (n+1-j)^{a+s-2} 
\lesssim (n+1)^{a+s-1}.
\]
Therefore
\[
\begin{split}
(II)
&\lesssim
\dfrac{1}{(n+1)^a} \sum_{m=n+1}^{2n} \abs{x_m}^2 (n+1)^{a+s-1}
=
(n+1)^{s-1} \sum_{m=n+1}^{2n} \abs{x_m}^2 \\
&\lesssim \sum_{m=n+1}^{2n} \abs{x_m}^2 (m+1)^{s-1}
\longrightarrow 0 \quad (\tn{as } n \to \infty).
\end{split}
\]
Finally, in (III), since $m>2n$ we have that
\[
\sum_{j=0}^{n} (n+1-j)^{a-1} (m+1-j)^{s-1}
\lesssim (m+1)^{s-1} \sum_{j=0}^{n} (n+1-j)^{a-1} \lesssim (m+1)^{s-1} (n+1)^{a}.
\]
Therefore
\[
(III)
\lesssim
\sum_{m=2n+1}^{\infty} \abs{x_m}^2 (m+1)^{s-1}
\longrightarrow 0 \quad (\tn{as } n \to \infty).
\]
Hence \eqref{eq lim Bs Ca bdd is 0} follows when $1-s < a \le 1$.
Finally, suppose that $1<a$. Then
\[
\dfrac{1}{(n+1)^a} \sum_{j=0}^{n} (n+1-j)^{a-1} \norm{B_s^j x}^2
\le
\dfrac{1}{n+1} \sum_{j=0}^n \norm{B_s^jx}^2
\longrightarrow 0 \quad (\tn{as } n \to \infty),
\]
since this is the case of $a=1$ in \eqref{eq lim Bs Ca bdd is 0} (already proved).
Note that \eqref{eq lim Bs Ca bdd is 0} implies
quadratical $(C,a)$-boundedness, so the proof is complete.
\end{proof}

This lemma allows us to prove the following more general result.

\begin{thm}\label{thm gen Bs Cbp-bdd}
Let $0<s<1$ and $1 \le q \le 2$. 
Then $B_s$ is $(C,b,q)$-bounded if and only if $b > q(1-s)/2$. 
Moreover, for $b > q(1-s)/2$ we have
\begin{equation}\label{eq gen Bs Cbp-bdd limit 0}
\lim_{n\to\infty}\frac{1}{k^{b+1}(n)}\displaystyle\sum_{j=0}^n k^{b}(n-j)\norm{B_s^jx}^q=0
\qquad \qquad
(\forall x \in H).
\end{equation}
\end{thm}

\begin{proof}
Note that $q = 2$ is precisely Lemma \ref{quadraticShift}. 
So we assume that $1 \le q < 2$. 
If $b = q(1-s)/2$, taking $x = e_n$, we get, as in \eqref{eq quadraticShift for 1-s}, that
\[
\frac{1}{k^{b+1}(n)}\displaystyle\sum_{j=0}^n k^{b}(n-j)\norm{B_s^j e_n}^q
\gtrsim
\log(n+2) \norm{e_n}^2
\]
for every $n$. Therefore $B_s$ is not $(C,q(1-s)/2,q)$-bounded,
and by Lemma~\ref{lemma Cap bdd implies Cbp bdd}
we get that $B_s$ is not $(C,b,q)$-bounded for $b < q(1-s)/2$.

Now suppose that $b > q(1-s)/2$. Then $b = qa/2$ for some $a>1-s$. 
Using H\"older's inequality as in the proof of Theorem \ref{th4.4}, we obtain
\[
\begin{split}
\dfrac{1}{(n+1)^b} \sum_{j=0}^{n} (n+1-j)^{b-1} \norm{B_s^j x}^q
&\lesssim
\left(
\dfrac{1}{(n+1)^a} \sum_{j=0}^{n} (n+1-j)^{a-1} \norm{B_s^j x}^2
\right)^{q/2} \\
&\xrightarrow[n \to \infty]{} 0,
\end{split}
\]
by Lemma \ref{quadraticShift}. Hence \eqref{eq gen Bs Cbp-bdd limit 0} follows.
\end{proof}

\begin{proof}[Proof of Theorem \ref{thm Ca to Cesaro bounded}]
Let $T \in \pweak{a}$ with $0<a<1$ and let $b > 1-a$.
By Theorem~\ref{thmBCHMc} and Theorem~\ref{thm can choose VD} (i),
$T$ is unitarily equivalent to a part of $(B_a \otimes I_\mathfrak{D}) \oplus S$.
Hence, by Lemma~\ref{lemma inherit by a part a tensor IR} (i),
it is enough to prove that $(B_a \otimes I_\mathfrak{D}) \oplus S$ is
quadratically $(C,b)$-bounded.
But this is immediate using
Lemma \ref{lemma inherit by a part a tensor IR} (ii) and (iii), 
and Lemmas \ref{lemma S is Cap bounded} and \ref{quadraticShift}.
\end{proof}

For the proof of Theorem~\ref{thm S appears iff}
we need the following lemma, which is in
the spirit of Lemma~\ref{lemma inherit by a part a tensor IR}.

\begin{lemma}\label{lemma inherit by a part a tensor IR 2}
The following holds.
\begin{itemize}
\item[\tn{(i)}] 
If $T$ satisfies \eqref{eq limit C-means are 0 general}, 
then any part of $T$ also satisfies \eqref{eq limit C-means are 0 general}.
\item[\tn{(ii)}] 
If $T_1$ and $T_2$ satisfy \eqref{eq limit C-means are 0 general}, 
then any direct sum $T_1 \dotplus T_2$ also satisfies 
\eqref{eq limit C-means are 0 general}.
\item[\tn{(iii)}] 
Let $T$ be a bounded linear operator on a Hilbert space. 
If $T$ satisfies \eqref{eq limit C-means are 0 general}, 
then the operator $T \otimes I_\CE$ also satisfies \eqref{eq limit C-means are 0 general}, 
where $I_\CE$ is the identity operator on some Hilbert space $\CE$.
\end{itemize}
\end{lemma}

\begin{proof}
(i) and (ii) are immediate.
For (iii) we use the same argument as in Lemma \ref{lemma inherit by a part a tensor IR} (iii)
and a simple application of Lebesgue's Dominated Convergence Theorem.
\end{proof}

\begin{proof}[Proof of Theorem~\ref{thm S appears iff}]
As in the proof of Theorem \ref{thm Ca to Cesaro bounded}, 
we have that $T$ is unitarily equivalent to
\[
(B_a \otimes I_\mathfrak{D}) \oplus S \, | \,  \mathcal{L},
\]
where $\mathcal{L}$ is a subspace of 
$(\mathcal{H}_{a} \otimes \mathfrak{D}) \oplus \mathcal{W}$ invariant by 
$(B_a \otimes I_\mathfrak{D}) \oplus S$.

Let us prove the circle of implications (i) $\Rightarrow$ (ii) $\Rightarrow$ (iii) $\Rightarrow$ (i).

Suppose that (i) is true. That is, $T$ is unitarily equivalent to
\[
(B_a \otimes I_\mathfrak{D}) \, | \,  \mathcal{L},
\]
where $\mathcal{L}$ is a subspace of $\mathcal{H}_{a} \otimes \mathfrak{D}$ 
invariant by $B_a \otimes I_\mathfrak{D}$. 
Then (ii) follows using Lemmas \ref{quadraticShift} and 
\ref{lemma inherit by a part a tensor IR 2}.

Suppose now that
\[
\liminf_{n \to \infty} \norm{T^n x} > 0
\]
for some $x \in H$. 
Then, obviously, $\norm{T^n x} > \ep > 0$ for every $n \ge 0$. 
Hence for this vector $x$ \eqref{eq limit C-means are 0 general} does not hold. 
Therefore we have proved that (ii) $\Rightarrow$ (iii).

Finally, suppose that the isometry $S$ appears in the minimal model.
Then for some vector $\ell = (\ell_1, \ell_2) \in \mathcal{L}$, 
its second component $\ell_2 \in \mathcal{W}$ is not $0$. 
Therefore
\[
\begin{split}
& \lim_{n \to\infty} \frac{1}{k^{b+1}(n)}\sum _{j=0}^n  k^{b}(n-j) 
\big\|((B_a \otimes I_\mathfrak{D})\oplus S)^j \ell \big\|^2 \\
&= 
\lim_{n \to\infty} \frac{1}{k^{b+1}(n)}\sum _{j=0}^n  k^{b}(n-j) 
\big\|(B_a \otimes I_\mathfrak{D})^j \ell_1 \oplus S^j \ell_2 \big\|^2 \\
&= 
\lim_{n \to\infty} \frac{1}{k^{b+1}(n)}\sum _{j=0}^n  k^{b}(n-j) 
\big\|(B_a \otimes I_\mathfrak{D})^j \ell_1 \big\|^2
+ \lim_{n \to\infty} \frac{1}{k^{b+1}(n)}\sum _{j=0}^n  k^{b}(n-j) \big\| S^j \ell_2 \big\|^2.
\end{split}
\]
The second limit is $\norm{\ell_2}^2 \neq 0$ because of \eqref{eq S is Cap bounded}. 
Hence we obtain that (iii) $\Rightarrow$ (i).
\end{proof}

\begin{rem}
In the same way, we get that if $T$ is an $a$-contraction and $0<a\le 1$, then
\[
\liminf_{n \to \infty} \norm{T^n x} \le \|x\|.
\]
In particular, this lower limit is finite for any $x$.
\end{rem}

Since $\pweak{1}$ is just the set of all contractions on $H$,
$T\in \pweak{1}$ iff $T^*\in \pweak{1}$. However, this is no longer
true for $a\in (0,1)$.

\begin{prop}
If $a\in (0,1)$, then there is an operator $T\in \pweak{a}$ such that $T^*\notin \pweak{a}$.
\end{prop}

\begin{proof}
Note that $B_a^*$ is a forward weighted shift such that $\|B_a^{*n}f_0\|\to\infty$ as
$n$ goes to $\infty$. So
$B_a\in \pweak{a}$, whereas its adjoint cannot belong to $\pweak{a}$, because
$B_a^*$ is not quadratically $(C,b)$-bounded
for any $b$ (see Lemma~\ref{quadraticShift}).
\end{proof}

It is natural to pose the following question.

\begin{que}
For which functions $\al$, satisfying Hypotheses~\ref{hypo-al}, is it true that
$T\in \pweak{\al}$ implies $T^* \in \pweak{\al}$?
\end{que}

It is so for $\al(t)=1-t$ and, more generally, for $\al(t)=1-t^n$, $n\ge 1$.
The authors do not know other examples.

\begin{rems}
\quad
\begin{itemize}
\item[\tn{(i)}]
If $T$ is an operator in $\pweak{a}$ with $0<a<1$, and $0<q<2$, then
by Theorem \ref{thm Ca to Cesaro bounded}
and Theorem \ref{th4.4}, it follows that $T$ is $(C,b,q)$-bounded
for all $b>\frac{q(1-a)}{2}$.
\item[\tn{(ii)}]
An $m$-isometry $T$, which is not an isometry, cannot be $(C,a,p)$-bounded, because
there are vectors $x$ such that the norms $\|T^nx\|$ go to infinity.
The possibility for these operators to have weaker ergodic properties, such as
the Ces\`aro boundedness and weak ergodicity, have been studied in \cite{BBMP19}.
\item[\tn{(iii)}]
Let $T$ be an operator in $\pweak{a}$ with $0<a<1$.
Using Theorem \ref{thm Ca to Cesaro bounded} (i) and Theorem \ref{th4.4}
(with $p=2$ and $q=1$) we
obtain that $T$ is $(C,b,1)$-bounded for every $b>(1-a)/2$.

By \cite[Corollary 3.1]{AB19}, we get that
$T$ is $(C,b)$-mean ergodic,
that is,
there exists 
\[
P_b x:=\lim_{n\to\infty}M^b_{T}(n)x,\quad x\in H.
\]
Therefore, by \cite[Theorem 3.3]{AGL19}, we have
\[
H={\rm {Ker}}(I-T)\oplus\overline{{\rm{Ran}}(I-T)}.
\]
In fact,
\[
{\rm {Ker}}(I-T)= {\rm{Ran}}P_{b}\ \text{ and }\  \overline{{\rm{Ran}}(I-T)}={\rm {Ker}}P_b.
\]
Also note that
\[
M^b_{T}(n) x=x \text{ for } x\in  {\rm {Ker}}(I-T), 
\text{ and  } 
\lim_{n\to\infty}M^b_{T}(n) x= 0 \text{ for } x\in  \overline{{\rm{Ran}}(I-T)}.
\]

Let now $0<\gamma<1,$ by \cite[Proposition 4.8 and Remark 4.9]{AGL19}, 
one can define a bounded operator $(I-T)^\ga$ by means 
of a certain functional calculus, and
\[
{\rm {Ker}}(I-T)={\rm {Ker}}(I-T)^{\gamma},\quad \overline{{\rm{Ran}}(I-T)}
=\overline{{\rm{Ran}}(I-T)^{\gamma}},
\]
with ${\rm{Ran}}(I-T)\subseteq {\rm{Ran}}(I-T)^{\gamma}$.
Furthermore if $\gamma < 1-b,$
for $x\in \overline{{\rm{Ran}}(I-T)},$
\[
x \in {\rm{Ran}}(I-T)^{\gamma} 
\Longleftrightarrow 
\sum_{n=1}^{\infty}\frac{1}{n^{1-\gamma}}T^n x  \ \text{ converges,}
\]
see \cite[Theorem  9.2]{AGL19}.
\item[\tn{(iv)}]
By \cite[Theorem 3.1]{AB19}, if $T$ is an operator in $\pweak{a}$
with $0<a<1$ and $b>(1-a)/2$, then
\[
\lim_{n \to \infty}\|M^{b}_{T}(n+1)-M^{b}_{T}(n)\|=0.
\]
\end{itemize}
\end{rems}

\section*{Acknowledgments}

The authors thank T. Bhattacharyya and N. Nikolski 
and D. Schillo for their useful remarks,
and A.~Bonilla for his advice concerning Theorem~\ref{thm Cap equals C1p}.
The first author has been partly supported by Project MTM2016-77710-P,
DGI-FEDER, of the MCYTS, Project E26-17R, D.G. Arag\'on, 
and Project for Young Researchers, Fundación Ibercaja and Universidad de Zaragoza, Spain.
The second author has been partially supported by
La Caixa-Severo Ochoa grant
(ICMAT Severo Ochoa project SEV-2011-0087, MINECO).
Both second and third authors acknowledge partial support by
Spanish Ministry of Science, Innovation and
Universities (grant no. PGC2018-099124-B-I00) and
the ICMAT Severo Ochoa project SEV-2015-0554 of the Spanish Ministry of Economy and
Competitiveness of Spain and the European Regional Development
Fund, through the ``Severo Ochoa Programme for Centres of Excellence
in R$\&$D''.
Both second and third authors also acknowledge
financial support from the Spanish Ministry of 
Science and Innovation, through the ``Severo Ochoa Programme for 
Centres of Excellence in R$\&$D'' (SEV-2015-0554) and from the Spanish 
National Research Council, through the ``Ayuda extraordinaria a 
Centros de Excelencia Severo Ochoa'' (20205CEX001).

\bibliographystyle{siam}
\bibliography{biblio_ABY_fin_2A}

\end{document}